\documentclass[11pt, oneside]{amsart}   	
\usepackage{geometry}                		
\usepackage{graphicx}				
\usepackage{amsaddr}								
\usepackage{amssymb}
\usepackage{color}
\usepackage{amsmath}
\usepackage{amsthm}
\usepackage{appendix}
\usepackage{esint}

\allowdisplaybreaks
\numberwithin{equation}{section}

\newcommand{\R}{\mathbb{R}}

\renewcommand{\S}{\mathbb{S}}

\newcommand{\Ric}{\operatorname{Ric}}
\newcommand{\Rm}{\operatorname{Rm}}

\newcommand{\id}{\operatorname{id}}
\newcommand{\loc}{\text{loc}}

\newcommand{\diam}{\operatorname{diam}}

\renewcommand{\L}{\mathcal{L}}
\newcommand{\T}{\mathcal{T}}

\newcommand{\dist}{\operatorname{dist}}

\newcommand{\Sing}{\mathcal{S}}

\theoremstyle{plain}
\newtheorem{theorem*}{Theorem}
\newtheorem{corollary*}{Corollary}
\newtheorem{question*}{Question}
\newtheorem{definition*}{Definition}
\newtheorem{claim*}{Claim}
\newtheorem{theorem}{Theorem}[section]
\newtheorem{lemma}[theorem]{Lemma}

\theoremstyle{definition}
\newtheorem{definition}[theorem]{Definition}
\theoremstyle{remark}
\newtheorem{remark}[theorem]{Remark}

\title[Smoothing $L^\infty$ metrics with nonnegative scalar curvature]{Smoothing $L^\infty$ Riemannian metrics with nonnegative scalar curvature outside of a singular set}
\author{Paula Burkhardt-Guim}
\address{Stony Brook University}
\email{paula.burkhardt-guim@stonybrook.edu}
\date{\today}

\begin{document}
\begin{abstract}
We show that any $L^\infty$ Riemannian metric $g$ on $\R^n$ that is smooth with nonnegative scalar curvature away from a singular set of finite $(n-\alpha)$-dimensional Minkowski content, for some $\alpha>2$, admits an approximation by smooth Riemannian metrics with nonnegative scalar curvature, provided that $g$ is sufficiently close in $L^\infty$ to the Euclidean metric. The approximation is given by time slices of the Ricci-DeTurck flow, which converge locally in $C^\infty$ to $g$ away from the singular set. We also identify conditions under which a smooth Ricci-DeTurck flow starting from a $L^\infty$ metric that is uniformly bilipschitz to Euclidean space and smooth with nonnegative scalar curvature away from a finite set of points must have nonnegative scalar curvature for positive times.
\end{abstract}
\maketitle

\section{Introduction}

It is natural to ask when a singular space satisfying a lower curvature bound in a weak sense can be approximated by smooth metrics satisfying the same lower curvature bound in the classical sense. In some cases, there are positive results of this type: for example, it is known that any $3$-dimensional polyhedral Alexandrov space with nonnegative curvature can be approximated by smooth $3$-manifolds with nonnegative sectional curvature \cite{LebedevaMatveevPetruninShevchishin15} (see also \cite{Spindeler14}). However, an Alexandrov space with curvature bounded below by $k$ need not, in general, be a Gromov--Hausdorff limit of Riemannian manifolds with sectional curvature bounded below by $k$ \cite{PetersenWilhelmZhu95}. In fact, it is conjectured that there exist Alexandrov spaces with curvature bounded below that admit no approximation by a sequence of Riemannian manifolds of fixed dimension with sectional curvature bounded below by $k$, for any $k\in \R$ (\cite{GuijarroKapovitch95}, \cite{Kapovitch04}).

Given these results, one could ask whether any analogous statements can be made for spaces with scalar curvature bounded below, and indeed the purpose of this paper is to prove an approximation result in this setting. The synthetic theory of lower scalar curvature bounds is less established than that of Alexandrov spaces (see, for instance, \cite{Sormani17}, \cite{Gromov14}, \cite{Li20-2}, \cite{LeeLeFloch15}, \cite{Bamler16}, \cite{PBG19}, \cite{LammSimon23} for some of the different notions of weak convergence and weak lower bounds that have been proposed for lower scalar curvature bounds). Nonetheless, examples of singular Riemannian metrics that are smooth with nonnegative scalar curvature outside of some set arise naturally in the study of spaces with scalar curvature bounded below. For instance, there are a number of singular versions of the Riemannian Positive Mass Theorem for metrics that are smooth with nonnegative scalar curvature away from a singular set (see, for example, \cite{Miao02}, \cite{McFeronSzekelyhidi12}, \cite{ShiTam18}, \cite{JiangShengZhang22}, \cite{LeeTam21}, \cite{ChuLeeZhu22}), in which the metric is assumed to be continuous across the singular set (often with some additional regularity), and geometric constraints on the singular set may be imposed. More generally, metrics on a punctured domain (\cite[Section 3.9]{Gromov21}, \cite{ChuLeeZhu24}), metrics with ``edge-type'' singularities \cite[Theorem 1.4]{LiMantoulidis19}, and cones over $\S^3$ \cite{BamlerChen23} are all examples of metrics that are smooth almost everywhere and that have been studied in the context of nonnegative scalar curvature.

In this paper we show that, under relatively mild assumptions, a Riemannian metric that is smooth with nonnegative scalar curvature almost everywhere admits an approximation by smooth metrics with globally nonnegative scalar curvature.

For a Riemannian metric $g$, let $R(g)$ denote the scalar curvature of $g$. Let $\delta$ denote the Euclidean metric on $\R^n$. We show the following:
\begin{theorem}\label{thm:SmoothApprox}
For all $\alpha>2$ and $n\geq 3$ there exists $\bar \varepsilon(\alpha, n)$ such that the following is true:

Suppose that $g$ is a measurable (with respect to the Euclidean measure) Riemannian metric on $\R^n$ such that $||g - \delta||_{L^\infty(\R^n)} < \bar \varepsilon$, where $|| \cdot||_{L^\infty(\R^n)}$ is measured with respect to $\delta$. Suppose that $g$ is smooth on $\R^n\setminus \Sing$, where $\Sing \subset \R^n$ is a set of finite $(n-\alpha)$-dimensional Minkowski content, and that $R(g)\geq 0$ on $\R^n\setminus \Sing$. Then $g$ admits an approximation by smooth Riemannian metrics with nonnegative scalar curvature, in the sense that there exists a smooth family of Riemannian metrics $(g_t)_{t>0}$ on $\R^n$ such that 
\begin{equation*}
R(g_t)\geq 0 \text{ for all } t>0
\end{equation*}
and
\begin{equation*}
g_t \xrightarrow[t\searrow 0]{C^{\infty}_{\loc}(\R^n\setminus \Sing)} g.
\end{equation*}
\end{theorem}
Note that in Theorem \ref{thm:SmoothApprox}, $g$ need not be continuous, and that, aside from the Minkowski content requirement, we do not impose geometric conditions on $\Sing$.

\begin{question*}
For $n\geq 3$, does there exist a $L^\infty$ metric on $\R^n$ that is uniformly bilipschitz to the Euclidean metric smooth outside of a singular set of finite $(n-2)$-dimensional Minkowski content, for which the conclusion of Theorem \ref{thm:SmoothApprox} fails? Does there exist such a metric for which the $(n-2)$-dimensional Minkowski content of the singular set is $0$?
\end{question*}
\begin{question*}
Suppose that in the setting of Theorem \ref{thm:SmoothApprox}, $\Sing$ has Hausdorff dimension equal to $n-\alpha$ for some $\alpha > 2$, rather than finite $(n-\alpha)$-dimensional Minkowski content. Does the conclusion of Theorem \ref{thm:SmoothApprox} still hold?
\end{question*}

The approximation in Theorem \ref{thm:SmoothApprox} is given by time slices of a solution to the Ricci-DeTurck flow, a geometric flow closely related to the Ricci flow that depends on a choice of ``background'' metric (see Section \ref{sec:preliminaries} for details; we take the background metric to be the Euclidean metric for most of this paper), and Theorem \ref{thm:SmoothApprox} follows immediately from:
\begin{theorem}\label{thm:smoothNNSCLinftySet}
For all $\alpha>2$ and $n\geq 3$ there exists $\bar \varepsilon(\alpha, n)$ such that the following is true:

Suppose that $g$ is a measurable metric on $\R^n$ such that $||g - \delta||_{L^\infty(\R^n)} < \bar \varepsilon$, where $|| \cdot||_{L^\infty(\R^n)}$ is measured with respect to $\delta$. Suppose that $g$ is smooth on $\R^n\setminus \Sing$, where $\Sing \subset \R^n$ is a set of finite $(n-\alpha)$-dimensional Minkowski content, and that $R(g)\geq 0$ on $\R^n\setminus \Sing$. Then there exists a smooth Ricci-DeTurck flow $(g_t)_{t\in (0, \infty)}$ with respect to the background metric $\delta$ such that
\begin{equation*}
R(g_t)\geq 0 \text{ for all } t>0
\end{equation*}
and
\begin{equation*}
g_t \xrightarrow[t\searrow 0]{C^{\infty}_{\loc}(\R^n\setminus \Sing)} g.
\end{equation*}
\end{theorem}

\begin{remark}
Given that the scalar curvature under Ricci flow is a supersolution to the heat equation (see (\ref{eq:Ricciscalarev})), Theorem \ref{thm:smoothNNSCLinftySet} may seem surprising by analogy: consider a negative Gaussian evolving by the heat equation on $\R^n\times(0,\infty)$, which tends to $0$ everywhere except at the origin as $t\searrow 0$. This example demonstrates that the conclusion of Theorem \ref{thm:smoothNNSCLinftySet} is false when $R(g_t)$ is replaced with a solution to the heat equation with respect to $\delta$. The key differences between this example and the statement of Theorem \ref{thm:smoothNNSCLinftySet} are that the evolution of the volume form under Ricci flow is also influenced by the scalar curvature (see (\ref{eq:RicciVolev})), and also that there is a positive source term in the evolution equation for the scalar curvature under Ricci flow (\ref{eq:Ricciscalarev}).
\end{remark}

\begin{remark}
We expect results analogous to Theorems \ref{thm:smoothNNSCLinftySet} and \ref{thm:SmoothApprox} to hold on manifolds for $\varepsilon$ perturbations of complete smooth metrics with bounded curvature, in view of \cite{Simon02} and \cite{PBG19}. The condition that the metric $g$ be $(1+\varepsilon)$-bilipschitz to a fixed complete smooth background metric of bounded curvature is used in two ways:
\begin{enumerate}
\item To guarantee the existence of a Ricci-DeTurck flow starting from $g$, as in Theorem \ref{thm:KL+}, and
\item To ensure an a priori bound of the form $R(g_t)\geq - c\varepsilon/t$ for some $c>0$, for all $t>0$, which in turn is used to derive an upper bound for a backwards heat kernel, as in the proofs of Theorem \ref{thm:nonstandardHK} and \cite[Theorem 2.3]{BamlerChen23}.
\end{enumerate}
\end{remark}
Interestingly, the second use seems to be somewhat inessential in the case that the singular set consists of finitely many points:
\begin{theorem}\label{thm:discretesetsmoothing}
Suppose that $\Sing\subset \R^n$ has finite $0$-dimensional Minkowski content. Suppose that $g$ is a measurable Riemannian metric on $\R^n$ that is smooth on $\R^n\setminus \Sing$ and satisfies $R(g)\geq 0$ on this region. Suppose that there exists a smooth Ricci-DeTurck flow $(g_t)_{t\in (0,T)}$, defined for some $T>0$, on $\R^n$ with respect to the background metric $\delta$, satisfying:
\begin{enumerate}
\item\label{item:discretesetconvergence} $g_t \xrightarrow[t\searrow 0]{C^2_\loc(\R^n\setminus \Sing)} g$,
\item there exists some $c>0$ such that for $k = 1,2,$ $|\nabla^k(g_t)|_{\delta}\leq c/t^{k/2}$, where $\nabla$ is taken with respect to $\delta$,
\item there exists some $b>0$ such that $g_t$ is $(1+b)$-bilipschitz to $\delta$ for all $t\in (0,T)$, and
\item\label{item:discretesetRlowerbound} there exists some $0 < c_0 < n/2$ such that for all $t\in (0,T)$, $R(g_t)\geq -c_0/t$.
\end{enumerate}
Then $R(g_t)\geq 0$ for all $t\in (0,T)$.
\end{theorem}
\begin{remark}
Any Ricci-DeTurck flow $(g_t)_{t\in (0,T)}$ satisfies a universal lower scalar curvature bound of the form given by item (\ref{item:discretesetRlowerbound}) with $c_0 = n/2$; see (\ref{eq:universalRbound}). Theorem \ref{thm:discretesetsmoothing} does not address the edge case $c_0 = n/2$.
\end{remark}

In a previous draft of this paper that was posted on the arXiv, we posed the following question concerning the sharpness of the $(1+\varepsilon)$-bilipschitz condition:
\begin{question*}\label{question:epsilonnecessary}
Is the $(1+\varepsilon)$-bilipschitz condition necessary? That is, are Theorems \ref{thm:SmoothApprox} and \ref{thm:smoothNNSCLinftySet} also true for metrics that are merely uniformly bilipschitz to some fixed complete smooth background metric of bounded curvature?
\end{question*}
Question \ref{question:epsilonnecessary} has since been answered by Cecchini -- Frenck -- Zeidler \cite[Theorem B]{CecchiniFrenckZeidler24}. They show that the $(1+\varepsilon)$-bilipschitz is indeed necessary: for certain $n\geq 8$ there exists a metric $g$ on $\R^n$ that is uniformly bilipschitz to the Euclidean metric and smooth with positive scalar curvature on $\R^n\setminus\{0\}$, but for which there exists no smooth family of Riemannian metrics $(g_t)_{t\in (0,T)}$ satisfying both 
\begin{equation*}
R(g_t)\geq 0 \text{ for all } t\in (0, T)
\end{equation*} 
and 
\begin{equation*}
g_t \xrightarrow[t\searrow 0]{C^2_\loc(\R^n\setminus\{0\})} g.
\end{equation*} 
In particular, Theorem \ref{thm:discretesetsmoothing} places restrictions on possible Ricci-DeTurck flows, with the background metric $\delta$, starting from these metrics.

Given the interest in $C^0$ metrics with scalar curvature bounded below in a synthetic sense (\cite{Gromov14}, \cite{Bamler16}, \cite{LeeTam21}, \cite{AntonelliFogagnoloNardulliPozzetta24}), we now record a version of Theorem \ref{thm:smoothNNSCLinftySet} that implicitly makes use of the weak local notion of lower scalar curvature bounds (\ref{eq:betaweak}) that was first introduced for $C^0$ metrics in \cite{PBG19}. However, note that while the local condition (\ref{eq:betaweak}) was originally introduced in the context of continuous metrics, the following theorem does not assume that the Ricci-DeTurck flow $(g_t)_{t>0}$ starts from continuous initial data.
\begin{theorem}\label{thm:betaNNSCLinftySet}
For all $\alpha>2$ and $n\geq 3$ there exist $\bar \varepsilon(\alpha, n)$, $\bar \beta = \bar \beta(\alpha, n)$ such that for all $0 < \beta < \bar \beta$ the following is true:

Suppose that $(g_t)_{t>0}$ is a smooth solution to the Ricci-DeTurck flow with $\delta$-background on $\R^n$, such that for $k = 0,1,2$ there exists some $\varepsilon \in (0, \bar \varepsilon)$ such that $||\nabla^k (g_t - \delta)||_{L^\infty(\R^n)} \leq \varepsilon/t^{k/2}$. Suppose that for some $\alpha > 2$, $\Sing \subset \R^n$ is a set of finite $(n-\alpha)$-dimensional Minkowski content, such that for some $0 < \beta < \bar \beta$, $(g_t)_{t>0}$ satisfies the $\beta$-weak condition (\ref{eq:betaweak}) at all $x\in \R^n\setminus \Sing$. Then $R(g_t)\geq 0$ for all $t>0$.
\end{theorem}
Similarly, the statement of Theorem \ref{thm:discretesetsmoothing} holds when the condition that $R(g)\geq 0$ on $\R^n\setminus\Sing$ is replaced with the condition that $(g_t)$ satisfies the $\beta$-weak condition on $\R^n\setminus\Sing$, when $\beta$ is sufficiently small depending on $c_0$ and $n$. A similar result to Theorem \ref{thm:betaNNSCLinftySet} was proven for continuous initial data on closed manifolds by Lee -- Tam in \cite[Corollary 4.2]{LeeTam21}. Their method of proof is different from the methods in this paper.

As an intermediary step to proving Theorems \ref{thm:betaNNSCLinftySet} and \ref{thm:discretesetsmoothing}, we first prove the following spatial bound for scalar curvature under Ricci-DeTurck flow:
\begin{theorem}\label{thm:SingSetSpatialLowerbound}
Suppose that $(g_t)_{t\in (0,1]}$ is a solution to the Ricci-DeTurck flow on $\R^n$, with respect to the Euclidean background metric $\delta$, such that $g_t$ is $(1 + b)$-bilipschitz to $\delta$ for some $b>0$, and for which there exists some $c>0$ such that $||\nabla^{k}(g_t)||_{C^0(\R^n)} \leq c/t^{k/2}$ for $k = 1,2$, where $\nabla$ is taken with respect to $\delta$. Let $\Sing\subset \R^n$. Suppose that there exists some $0 < \beta < 1/2$ such that for all $x\in \R^n\setminus \Sing$, $g_t$ satisfies the $\beta$-weak condition (\ref{eq:betaweak}) at $x$ with the lower bound $\kappa_0$, for some $\kappa_0 \in \R$. Then, for any $0 < \eta \leq (1-2\beta)/2$ there exists $\bar C(c,n,\kappa_0)$, $\bar D(c,n)$, $\bar c(c,n)$ such that for all $0< t \leq 1$ and $x\notin \T(\Sing, \bar ct^\eta)$ we have
\begin{equation*}
R(x,t) \geq \kappa_0 -\frac{\bar C}{t}\exp\left(-\frac{d^2_{\delta}(x, \Sing)}{\bar Dt^{1-2\beta}}\right),
\end{equation*}
where $\T(\Sing, \bar ct^{\eta}) = \{y\in \R^n: d_{\delta}(y, \Sing) < \bar ct^\eta\}$.
\end{theorem}

We now describe the organization of this paper. In Section \ref{sec:preliminaries} we establish some preliminary material concerning Ricci and Ricci-DeTurck flow. In particular, we introduce two heat kernels that will be used in subsequent sections, and record some estimates for these kernels. In Section \ref{sec:spatialLowerBound} we use the fact that scalar curvature is a super solution to the heat equation under Ricci-DeTurck flow and exponential heat kernel estimates to prove Theorem \ref{thm:SingSetSpatialLowerbound}. In Section \ref{sec:uniformLowerBound} we upgrade the spatial lower bound from Section \ref{sec:spatialLowerBound} to a uniform nonnegative lower bound, using a backwards heat kernel introduced by Bamler -- Chen \cite{BamlerChen23}, in order to prove Theorems \ref{thm:betaNNSCLinftySet}, \ref{thm:smoothNNSCLinftySet},  and \ref{thm:discretesetsmoothing}.

\subsection*{Acknowledgements} I am grateful to a number of people who have suggested versions of this problem to me. They are: Richard Bamler, Bruce Kleiner, Chao Li, and Christos Mantoulidis. In addition, I would like to thank Bruce Kleiner for many helpful conversations about both the mathematics and the writing of this paper. I would also like to thank Richard Bamler for helpful discussions, especially concerning \cite{BamlerChen23}. Finally, I am grateful to Vitali Kapovitch for a very illuminating correspondence about Alexandrov spaces. I am also grateful Man Chun Lee for a clarifying conversation about the role of the $(1+\varepsilon)$-bilispchitz condition in an earlier version of this paper, and to Rudolf Zeidler for sending me a draft of \cite{CecchiniFrenckZeidler24}.

This material is based upon work supported by the National Science Foundation under Award No. DMS $2103145$.

\section{Preliminaries}\label{sec:preliminaries}
We fix the following notation throughout: Let $\delta$ denote the Euclidean metric. For $\Sing \subset \R^n$ and $r>0$, let 
\begin{equation*}
\T(\Sing, r) = \{y\in \R^n: d_\delta(x,y) < r \text{ for some } x\in \Sing\}.
\end{equation*}

Recall that, for $0 \leq m \leq n$ and $\Sing \subset \R^n$, the $m$-dimensional Minkowski content of $\Sing$ is
\begin{equation*}
\lim_{r\to 0^+} \frac{|\{x: d_\delta(x,\Sing) < r\}|}{\omega(n-m)r^{n-m}},
\end{equation*}
where here $|\cdot|$ denotes $n$-dimensional Lebesgue measure (with respect to $\delta$) and $\omega(n-m)$ denotes the Euclidean volume of the $(n-m)$-dimensional unit ball, provided that this limit exists.

\begin{lemma}\label{lemma:CCSingularSet}
For $n\geq 3, \alpha \geq 0$, let $\Sing \subset \R^n$ be a set of finite $(n-\alpha)$-dimensional Minkowski content. Then there exists some compact $K\subset \R^n$ such that $\Sing\subset K$.
\end{lemma}
\begin{proof}
Since $\Sing$ has finite $(n-\alpha)$-dimensional Minkowski content, by definition there exists some $c>0$ such that for all sufficiently small $r>0$, 
\begin{equation*}
|\T(\Sing, r)| \leq cr^{\alpha}.
\end{equation*}
Let $\{x_i\}_{i=1}^{\gamma}$ be a maximal set of points in $\T(\Sing, r)$ such that the balls  $\{B(x_i, r/2)\}_{i=1}^{\gamma}$ are pairwise disjoint. Then $\{B(x_i, r)\}_{i=1}^{\gamma}$ covers $\T(\Sing, r)$. Moreover, each $B(x_i, r/2)\subset \T(\Sing, 1.5r)$ so if $r>0$ is sufficiently small, we have
\begin{equation*}
\gamma \omega(n)(\tfrac{r}{2})^n = \sum_{i=1}^{\gamma} |B(x_i, \tfrac{r}{2})| \leq |\T(\Sing, 1.5r)| \leq  cr^{\alpha},
\end{equation*}
so $\gamma \leq c'r^{\alpha - n}$ for some $c'>0$, that is, $\gamma$ is finite. Thus, for $r$ sufficiently small,
\begin{equation*}
K_r := \bigcup_{i=1}^{\gamma} \overline{B(x_i, r)}
\end{equation*}
is a compact set containing $\T(\Sing, r)$ and hence $\Sing$.
\end{proof}

\subsection{Ricci and Ricci-DeTurck flow, and evolution of geometric quantities}

If $M$ is a smooth manifold and $(\tilde g_t)_{t\in (0, T)}$ is a smooth family of Riemannian metrics on $M$, recall that $\tilde g_t$ evolves by Ricci flow if
\begin{equation}\label{eq:RF}
\partial_t \tilde g_t = -2\Ric(\tilde g_t).
\end{equation}
We use the notation $\tilde g_t$ to distinguish this flow from the Ricci-DeTurck flow, which we use more often in this paper, and which we will denote by $g_t$. The Ricci-DeTurck flow, introduced by DeTurck in \cite{DeTurck83}, is a strongly parabolic flow that is related to the Ricci flow by pullback via a family of diffeomorphisms. More specifically, given a fixed smooth background metric $\bar g$, we define the following operator, which maps symmetric $2$-forms on $M$ to vector fields (\cite[(A.7)]{BamlerKleiner22}):
\begin{equation}\label{eq:Xoperator}
X_{\bar g}(g):= \sum_{i=1}^n(\nabla^{\bar g}_{e_i}e_i - \nabla^{g}_{e_i}e_i),
\end{equation}
where $\{e_i\}_{i=1}^n$ is any local orthonormal frame with respect to $g$. Then the Ricci-DeTurck equation is
\begin{equation}\label{eq:RDTF}
\partial_t g(t) = -2\Ric(g(t)) - \L_{X_{\bar g(t)}(g(t))}g(t),
\end{equation}
where $\bar g(t)$ is a smooth background Ricci flow. In this paper we often work with Ricci-DeTurck flows on $\R^n$ with respect to a Euclidean background, so we take $\bar g(t)\equiv \delta$ and (\ref{eq:RDTF}) becomes
\begin{equation}\label{eq:RDTFeucl}
\partial_t g(t) = -2\Ric(g(t)) - \L_{X_{\delta}(g(t))}g(t).
\end{equation}
\begin{remark}\label{rmk:parabolicscaling}
If $g_t$ solves (\ref{eq:RDTF}) with respect to some background Ricci flow $\bar g_t$ for $t\in (a,b)$, then, for all $\lambda >0$, the parabolically rescaled flow $g'_t:= \lambda g_{t/\lambda}$ solves (\ref{eq:RDTF}) with respect to the background Ricci flow $\bar g'_t := \lambda \bar g_{t/\lambda}$ for $t\in (\lambda a, \lambda b)$. If $\bar g_t \equiv \delta$ and $g(x,t)$ solves (\ref{eq:RDTFeucl}), then $g'(x,t) := g(x/\sqrt{\lambda}, t/\lambda)$ also solves (\ref{eq:RDTFeucl}).
\end{remark}

In what follows, we work with solutions $g_t$ to the Ricci-DeTurck equation (\ref{eq:RDTF}) with respect to some given background Ricci flow $\bar g_t$, such that $(g_t)_{t\in (0,T)}$ and $(\bar g_t)_{t\in (0,T)}$ are uniformly bilipschitz to some fixed smooth background metric $\bar g$ (often we take $\bar g = \bar g_0$ and in subsequent sections we take $\bar g_0 = \delta$, so that $\bar g_t = \delta$ for all $t$). In this setting, unless otherwise stated, all balls, distances, and norms are measured with respect to $\bar g$ and all derivatives of $g_t$ are taken using $\bar g_t$.

If $g_t$ solves (\ref{eq:RDTFeucl}) then $h_t(x) = h(x,t) := g_t(x) - \delta$ evolves by (\cite[(4.4)]{KochLamm12})
\begin{equation}\label{eq:hevolution}
\begin{split}
\partial_t h_{ij} &= \Delta h_{ij} + \frac{1}{2}(\delta + h)^{pq}(\delta + h)^{m\ell}\big( \nabla_i h_{pm}\nabla_j h_{q\ell} + 2\nabla_m h_{ip}\nabla_q h_{jm} -2\nabla_m h_{ip}\nabla_{\ell}h_{jp} 
\\& - 2\nabla_p h_{i\ell}\nabla_j h_{qm} -2\nabla_i h_{pm}\nabla_q h_{j\ell} \big) - \nabla_p((\delta + h)^{pq})\nabla_q h_{ij}
\\& + \nabla_p \big( ((\delta + h)^{pq} - \delta^{pq})\nabla_q h_{ij} \big)
\\&=: \Delta h_{ij} + Q^0[h] + \nabla^*Q^1[h]
\\& \text{ where}
\\& Q^0[h] = \frac{1}{2}(\delta + h)^{pq}(\delta + h)^{m\ell}\big( \nabla_i h_{pm}\nabla_j h_{q\ell} + 2\nabla_m h_{ip}\nabla_q h_{jm} -2\nabla_m h_{ip}\nabla_{\ell}h_{jp} 
\\& - 2\nabla_p h_{i\ell}\nabla_j h_{qm} -2\nabla_i h_{pm}\nabla_q h_{j\ell} \big) - \nabla_p((\delta + h)^{pq})\nabla_q h_{ij} = \nabla h * \nabla h,
\\& \nabla^*Q^1[h] = \nabla_p \big( ((\delta + h)^{pq} - \delta^{pq})\nabla_q h_{ij} \big) = \nabla (h * \nabla h),
\end{split}
\end{equation}
and where $\Delta$ denotes the usual Euclidean Laplacian. If $\bar K(x, t; y,s)$ is the Euclidean heat kernel, then the integral equation corresponding to (\ref{eq:hevolution}) with initial data $h_0$ is (see \cite[Section 2.1]{KochLamm12})
\begin{equation}\label{eq:Integraleq}
h(x,t) = \int_{\R^n} \bar K(x,t; y, 0)h_0(y)dy + \int_0^t \int_{R^n} \bar K(x,t;y,s) (Q^0[h](y,s) + \nabla^*Q^1[h](y,s))dyds.
\end{equation}

As mentioned, if $g(t)$ solves (\ref{eq:RDTF}) then it is related to a Ricci flow via pullback by diffeomorphisms. More precisely, if $g(t)$ solves (\ref{eq:RDTF}) on $(0,T)$ and $(\chi_t)_{t\in (0, T)}: M\to M$ is the family of diffemorphisms satisfying
\begin{equation}\label{eq:diffeoseq}
\begin{cases}
X_{\bar g(t)}(g(t))f &= \frac{\partial}{\partial t}(f\circ\chi_t) \text{ for all } f\in C^\infty(M)\\
\chi_{t_0} &= \id,
\end{cases}
\end{equation}
for some $t_0\in (0,T)$, then $\tilde g(t):= \chi_t^*g(t)$ solves (\ref{eq:RF}) and satisfies $\tilde g(t_0) = g(t_0)$ (when $g_t$ is a smooth solution to (\ref{eq:RDTF}) on $[0,T)$ then typically we take $t_0 = 0$, but that is not the setting addressed in this paper).

If $\tilde g_t$ evolves by the Ricci flow, the scalar curvature of $\tilde g_t$ satisfies (\cite[(2.5.5)]{Topping06})
\begin{equation}\label{eq:scalareveq}
\partial_t R(\tilde g_t) = \Delta^{\tilde g_t} R(\tilde g_t) + 2|\Ric|^2(\tilde g_t)
\end{equation}
and hence (\cite[(2.5.6), (2.5.7)]{Topping06})
\begin{equation}\label{eq:Ricciscalarev}
\partial_t R(\tilde g_t) \geq \Delta^{\tilde g_t} R(\tilde g_t) + \frac{2}{n}R(\tilde g_t)^2
\end{equation}
and the volume form of $\tilde g_t$ evolves by
\begin{equation}\label{eq:RicciVolev}
\partial_t d\tilde g_t = -R(\tilde g_t)d\tilde g_t.
\end{equation}
If $g_t$ evolves by Ricci-DeTurck flow (\ref{eq:RDTF}) then the scalar curvature of $g_t$ satisfies (\cite[(2.24)]{PBG19})
\begin{equation}\label{eq:RDTFscalarev}
\partial_t R(g_t)  \geq \Delta^{g_t} R(g_t) - \langle X_{\bar g_t}(g_t), \nabla R(g_t)\rangle + \frac{2}{n}R(g_t)^2,
\end{equation}
where $X$ is as in (\ref{eq:Xoperator}). It follows from the maximum principle (see, for example, \cite[Corollaries 3.2.2 and 3.2.5]{Topping06}) that 
\begin{equation}\label{eq:universalRbound}
R(g_t)\geq -n/(2t)
\end{equation} 
for all $t>0$ for which the flow is defined, and that if $t_1 < t_2$ and $R(g_{t_1}) \geq \alpha$ for some $\alpha \in \R$, then $R(g_{t_2}) \geq \alpha$ as well.

We will eventually make use of the following condition on the scalar curvature of time slices of the Ricci-DeTurck flow as $t\searrow 0$, first introduced in \cite{PBG19}:
\begin{definition}
Let $M$ be a smooth manifold and $(g_t)_{t\in (0, T]}$ be a smooth solution to the Ricci-DeTurck flow, with respect to a smooth background Ricci flow $(\bar g_t)_{t\in [0,T]}$. Suppose that $(\bar g_t)_{t\in [0,T]}$ is uniformly bilipschitz to a fixed smooth Riemannian metric $\bar g$ on $M$. For $\beta \in (0, 1/2)$, we say that $(g_t)_{t\in (0,T]}$ satisfies the $\beta$-weak condition at $x$ with the lower bound $\kappa_0$ for some $\kappa_0 \in \R$ if 
\begin{equation}\label{eq:betaweak}
\inf_{C>0}\left(\liminf_{t\searrow 0}\left(\inf_{B_{\bar g}(x, Ct^\beta)}R^{g_t}\right)\right) \geq \kappa_0.
\end{equation}
\end{definition}

We will show that there exists a Ricci-DeTurck flow for which this condition holds, and use this condition to deduce a global lower bound for the scalar curvature of time-slices of the Ricci-DeTurck flow, for positive times. Note that when $\kappa_0 = 0$ the condition (\ref{eq:betaweak}) is invariant under the parabolic rescaling described in Remark \ref{rmk:parabolicscaling}. 
 
Because the Ricci-DeTurck flow is strongly parabolic, it has a regularizing effect on the initial data:
\begin{theorem}\label{thm:KL+}
There exists $\varepsilon = \varepsilon(n)$ and $C = C(n)>0$ such that the following is true:

 Suppose that $g$ is a measurable Riemannian metric on $\R^n$ such that $|| g - \delta||_{L^\infty(\R^n)} < \varepsilon$, where $\delta$ denotes the Euclidean metric, and $|| \cdot||_{L^\infty(\R^n)}$ is measured with respect to $\delta$. Suppose that $g$ is smooth on $\R^n\setminus \Sing$, where $\Sing$ is some subset of $\R^n$. Then there exists a solution $(g_t)_{t>0}$ to the Ricci-DeTurck flow (with respect to the background metric $\delta$) such that:
 \begin{enumerate}
 \item $g_t$ is smooth for all $t>0$,
 \item for all $t>0$,
 \begin{equation}
 || g_t - \delta||_{C^0(\R^n)} \leq C|| g - \delta||_{L^\infty(\R^n)},
 \end{equation}
\item for all $k\in \mathbb{N}$ there exists $c(k)>0$ such that for all $t>0$,
 \begin{equation}\label{eq:RDTFderivests}
 || \nabla^k g_t||_{C^0(\R^n)}\leq \frac{c(k)|| g - \delta ||_{L^\infty}(\R^n)}{t^{k/2}},
 \end{equation}
 and
 \item $g_t$ converges smoothly to $g$ as $t\searrow 0$ on compact subsets of $\R^n\setminus \Sing$.
 \end{enumerate}
\end{theorem}

The existence of a smooth solution and the derivative estimates stated in Theorem \ref{thm:KL+} are due to Koch -- Lamm (cf. \cite[Theorem 4.3]{KochLamm12}), by taking $g_t := h_t + \delta$:
\begin{theorem}[\cite{KochLamm12}]\label{theorem:KL}
There exists $\varepsilon = \varepsilon(n)$ and $C = C(n)>0$ such that the following is true:

 Suppose that $g$ is a measurable Riemannian metric on $\R^n$ such that $|| g - \delta||_{L^\infty(\R^n)} < \varepsilon$, where $\delta$ denotes the Euclidean metric, and $|| \cdot||_{L^\infty(\R^n)}$ is measured with respect to $\delta$. Then there exists a solution $(h_t)_{t>0}$ to the integral equation (\ref{eq:Integraleq}) with initial data $g-\delta$ such that
 \begin{enumerate} 
 \item $|| g_t - \delta||_{X_\infty} \leq C|| g-\delta||_{L^\infty(\R^n)}$, where
 \begin{equation}
 \begin{split}
 ||h||_{X_T} & := \sup_{0 < t < T}|| h||_{L^\infty(\R^n)} 
 \\& + \sup_{x\in \R^n}\sup_{0 < r < \sqrt{T}}\left(r^{-n/2}|| \nabla h||_{L^2(B(x,r)\times(0,r^2))} + r^{2/(n+4)}|| \nabla h||_{L^{n+4}(B(x,r)\times (r^2/2, r^2))}\right),
 \end{split}
 \end{equation}            
 \item $g_t$ is smooth for all $t>0$, and for all $k\in \mathbb{N}$ there exists $c(k)>0$ such that for all $t>0$,                                                                 
 \begin{equation}\label{eq:RDTFderivests}
 || \nabla^k g_t||_{C^0(\R^n)}\leq \frac{c(k)|| g - \delta ||_{L^\infty}(\R^n)}{t^{k/2}},
 \end{equation}
and
\item $h_t$ is the unique solution of (\ref{eq:Integraleq}) with initial data $g-\delta$ in $\{h_t : || h_t ||_{X} \leq C\varepsilon\}$.
 \end{enumerate}
\end{theorem}
\begin{remark}\label{rmk:Xnormscalinginvariance}
Note that the norm $|| \cdot||_{X_\infty}$ is invariant under the parabolic scaling described in Remark \ref{rmk:parabolicscaling}.
\end{remark}

It remains to address the convergence of the solution $g_t$ as $t\searrow 0$. This is due to the following result, which is proven in Appendix \ref{appendix:convergence}:
\begin{lemma}\label{lemma:SmoothConvergenceae}
Let $\varepsilon'$ be as in Lemma \ref{lemma:smoothRDTFderivbounds} and $\varepsilon$ be as in Theorem \ref{theorem:KL}. There exists some positive $\varepsilon'' = \varepsilon''(n) \leq \varepsilon' \leq \varepsilon$ such that the following is true:

Let $g$ be a $L^\infty$ metric on $\R^n$ that is $(1+\varepsilon'')$-bilipschitz to $\delta$. Suppose that $g$ is smooth on $\R^n\setminus \Sing$, where $\Sing$ is some bounded subset of $\R^n$. Then for any bounded open set $U$ such that $\dist(\overline{U}, \Sing)>0$, any $N\in (0, \infty)$, and any $k\in \mathbb{N}$ there exists $C_k := C_k(n, N, g, \overline{U}) >0$ and $T_k = T_k(n, N, g, \overline{U})>0$ such that if $|| g - \delta||_{C^k(\overline{U})} \leq N$ then for all $t\in [0, T_k]$ we have $|| \nabla^k g_t||_{C^0(\overline{U})} \leq C_k$, where $g_t := h_t + \delta$ and $(h_t)_{t>0}$ is the solution to the integral equation (\ref{eq:Integraleq}) with initial data $g-\delta$ given by Theorem \ref{theorem:KL}.
\end{lemma}

\begin{proof}[Proof of Theorem \ref{thm:KL+}]
Let $\varepsilon''$ be the constant from Lemma \ref{lemma:SmoothConvergenceae}. Set $\varepsilon = \varepsilon''$. Let $g$ be as in the statement of the theorem. By Theorem \ref{theorem:KL} there exists a solution $(h_t)_{t>0}$ to the integral equation (\ref{eq:Integraleq}) with initial data $g-\delta$. For $t>0$ let $g_t := h_t + \delta$, so that $(g_t)_{t>0}$ solves (\ref{eq:RDTFeucl}) for $t>0$. Let $K\subset \R^n\setminus \Sing$ be a compact subset, so that $r:= \dist(K,\Sing) >0$. Note that $\{B(x,r/2): x\in K\}$ is an open cover of $K$, so there exists a finite subcover $\{B(x_1, r/2), \ldots, B(x_m, r/2)\}$ for some $x_1, \ldots, x_m \in K$. Let 
\begin{equation*}
U := \cup_{i=1}^{m} B(x_i, r/2).
\end{equation*}
Then $U$ is a bounded open set and $\dist(\overline{U}, \Sing) \geq \dist(K, \Sing) - r/2 = r/2 >0$. Now fix $k\in \{0\}\cup \mathbb{N}$. By Lemma \ref{lemma:SmoothConvergenceae}, for $m = 1,\ldots, k+2$ there exists $C_{m} > 0$ and $T_{m}>0$ such that $||\nabla^{m} g_t||_{C^0(\overline{U})} \leq C_m$ for all $t\in [0, T_{m}]$. Let $T = \min\{T_1, \ldots, T_m\}$. Then, for all $x\in K$ and $t\in (0, T)$,
\begin{equation*}
|\nabla^k g_t(x) - \nabla^k g(x)| \leq \int_0^t |\partial_s \nabla^k g_s(x)| ds \leq C(C_1, \ldots, C_{k+2})t,
\end{equation*}
where in addition to Lemma \ref{lemma:SmoothConvergenceae} we have used (\ref{eq:hevolution}), (\ref{eq:higherorderQ0}), and (\ref{eq:higherorderQ1}). Therefore,
\begin{equation*}
|| \nabla^k g_t - \nabla^k g||_{C^0(K)} \xrightarrow[t\searrow 0]{} 0.
\end{equation*}
\end{proof}

\subsection{Heat kernel estimates}
Throughout, let $X_{\bar g_t}(g_t)$ denote the usual Ricci-DeTurck vector field, as in (\ref{eq:Xoperator}).

The \emph{standard heat kernel} for the Ricci-DeTurck flow, $\Phi(x,t;y,s)$, is \cite[Definition 23.1, 24.1]{ChowEtAl10} the minimal positive solution of
\begin{equation}\label{eq:RDTFheat}
\begin{cases}
\partial_t \Phi(x,t;y,s) &= \Delta_{g_t, x}\Phi(x,t;y,s) - \partial_{X_{\bar g_t}(g_t)}\Phi(x,t;y,s)\\
\partial_s \Phi(x,t;y,s) &= - \Delta_{g_s, y}\Phi(x,t;y,s) + \partial_{X_{\bar g_s}(g_s)}\Phi(x,t;y,s) +R^{g_s}\Phi(x,t;y,s)\\
K(\cdot, t;y,s) &\to \delta_y \text{ as } t\searrow s\\
K(x, t;\cdot,s) &\to \delta_x \text{ as } s\nearrow t.
\end{cases}
\end{equation}

\begin{theorem}[cf. \cite{Bamler16}]\label{thm:standardRDTFHK}

For any $0 < c < \infty$ and $n\in \mathbb{N}$ there exist constants $C_1 = C_1(n,c,b_1, b_2 v_0)$, $D_1 = D_1(n,c,b_1, b_2)$, $C_2 = C_2(n,c,b_1, b_2, v_0, \kappa)$, and $D_2 = D_2(n,c,b_1, b_2, \kappa)$ such that the following is true:

Let $(M^n, \bar g)$ be a smooth Riemannian manifold such that $\Ric(\bar g) \geq \kappa$ for some $\kappa\in \R$ and such that $\inf_{x\in M, 0 < r < 1}|B(x, r)|_{\bar g}/r^n > v_0$ for some $v_0 > 0$. Suppose that $(\bar g_t)_{t\in [0,1]}$ is a smooth Ricci flow on $M$ such that for some $b_1>0$ and for all $t\in [0,1]$, $\bar g_t$ is $(1+b_1)$-bilipschitz to $\bar g$. Let $0 < \theta \leq 1$ and suppose that $(g_t)_{t\in [\tfrac{\theta}{2}, \theta]}$ is a smooth solution to the Ricci-DeTurck equation on $M$ with respect to the background Ricci flow $\bar g_t$, such that for some $b_2>0$, $g_t$ is $(1+b_2)$-bilipschitz to $\bar g_t$ for $t\in [\tfrac{\theta}{2}, \theta]$. Also suppose that for all $t\in [\tfrac{\theta}{2}, \theta]$,  $|(\nabla^{\bar g_t})^{k}(g_t)|_{\bar g_t} \leq c/t^{k/2}$ for $k = 1,2$. Let $\Phi(x,t;y,s)$ denote the heat kernel for the Ricci-DeTurck flow on $\R^n\times [\tfrac{\theta}{2}, \theta]$ as in (\ref{eq:RDTFheat}). Then, for all $x,y\in \R^n$ and $s,t\in [\tfrac{\theta}{2}, \theta]$ with $s < t$, we have
\begin{equation}\label{eq:standardRDTFHKptwisebound}
\Phi(x,t; y,s) \leq \frac{C_1}{(t-s)^{n/2}}\exp\left(-\frac{|x-y|^2_{\bar g}}{D_1(t-s)}\right),
\end{equation}
\begin{equation}
\int_{M}\Phi(x,t;y,s)d_{g_s}(y) = 1,
\end{equation}
and
\begin{equation}\label{eq:standardHKintegraloutsideball}
\int_{M\setminus B_{\bar g}(x, r)} \Phi(x,t;y,s)d_{g_s}(y) \leq C_2\exp\left(-\frac{r^2}{D_2(t-s)}\right)
\end{equation}
\end{theorem}
\begin{proof}
The proof is similar to \cite[Lemma 2.9, Corollary 2.10]{PBG19} and \cite[Lemma 3 and p.3]{Bamler16}. The estimate (\ref{eq:standardHKintegraloutsideball}) is due to the super-exponential decay (\ref{eq:standardRDTFHKptwisebound}) and the fact that $g_s$ is uniformly bilipschitz to $\bar g$, which has a lower Ricci bound.
\end{proof}

\begin{theorem}\label{thm:nonstandardHK}
Let $(M^n, \bar g)$ be a smooth Riemannian manifold such that $\Ric(\bar g) \geq \kappa$ for some $\kappa\in \R$ and such that $\inf_{x\in M, 0 < r < 1}|B(x, r)|_{\bar g}/r^n > v_0$ for some $v_0 > 0$. Suppose that $(\bar g_t)_{t\in [0,1]}$ is a smooth Ricci flow on $M$ such that for some $b_1>0$ and for all $t\in [0,1]$, $\bar g_t$ is $(1+b_1)$-bilipschitz to $\bar g$. Suppose that $(g_t)_{t\in (0,1]}$ is a smooth solution to the Ricci-DeTurck flow on $M$ with respect to the background Ricci flow $\bar g_t$, such that for some $b_2>0$, $g_t$ is $(1+b_2)$-bilipschitz to $\bar g_t$, for all $t\in (0,1]$. Suppose that there exists $c>0$ such that $|\Rm(g_t)|\leq c/t$ for all $t\in (0,1]$.

Then, for any $p\in \R^n$, there exists a backwards heat kernel $u \in C^{\infty}(M \times (0,1))$ satisfying
\begin{equation}\label{eq:BackwardsHK}
\begin{cases}
-\partial_t u &= \Delta_{g_t} u - \partial_{X_{\bar g_t}(g_t)} u - (1-\tfrac{2}{n})R^{g_t}u\\
u(\cdot,t) & \xrightarrow[]{t \nearrow 1} \delta_p.
\end{cases}
\end{equation}
Moreover, there exists $C = C(n, c)$ such that if $0 < c_0 \leq n/2$ is some constant such that $ R(g_t)\geq -c_0/t$ for all $t\in (0, 1]$, then, for $0 < t \leq 1/2$,
\begin{equation}\label{eq:universalBackwardsHKbound}
u(\cdot, t) \leq Ct^{-(1-2/n)c_0}.
\end{equation}
\end{theorem}
\begin{proof}
First consider the Ricci flow $\tilde g_t := \chi_t^* g_t$, defined for $t\in (0, 1]$, where $\chi_t$ is the family of diffeomorphisms given by
\begin{equation*}
\begin{cases}
\partial_t \chi_t &= X_{\bar g_t}(g_t)\\
\chi_1 &= \id.
\end{cases}
\end{equation*}

We first show that there exists a backwards heat kernel $\tilde u$ solving
\begin{equation*}
\begin{cases}
-\partial_t \tilde u &= \Delta_{\tilde g_t} \tilde u - (1-\tfrac{2}{n})R^{\tilde g_t}\tilde u\\
\tilde u(\cdot,t) & \xrightarrow[t \nearrow 1]{} \delta_p.
\end{cases}
\end{equation*}
The existence of $\tilde u$ and a corresponding estimate of the form (\ref{eq:universalBackwardsHKbound}) for $\tilde u$ are essentially as in the proof of \cite[Theorem 2.3]{BamlerChen23}: $\tilde g_t$ has uniformly bounded curvature on intervals of the form $[a, 1]$, for $a >0$ (where the uniform curvature bound depends on $a$). Then, by \cite[Theorem 24.40]{ChowEtAl10}, on any such interval there exists a unique smooth positive fundamental solution $\tilde H^a(x,t;y,s)$ for the operator $\partial_t - \Delta_{x,t} + \tfrac{2}{n}R$, for $t,s\in [a, 1]$ with $s<t$. 
By uniqueness of the fundamental solution on $[a,1]$, it follows that there is a fundamental solution $\tilde H(x,t;y,s)$ for $\partial_t - \Delta_{x,t} + \tfrac{2}{n}R$ defined for $t,s\in (0, 1]$ with $s<t$.

Also, since $(\tilde g_t)_{t\in (0,1]}$ has uniformly bounded curvature for $t\in [\tfrac{1}{2}, 1]$ and $\tilde g_1 = g_1$ is uniformly bilipschitz to $\bar g$, it follows (see, for instance, \cite[Lemma 5.3.2]{Topping06}) that $\tilde g_t$ is uniformly bilipschitz (with a time-independent bilipschitz constant) to $\bar g$ on $[\tfrac{1}{2}, 1]$. Therefore, by \cite[Theorem 26.25]{ChowEtAl10}, there exist uniform constants $C_1(n, c, b, v_0)$ and $C_2(n, c, b)$ such that such that for $t,s\in [\tfrac{1}{2},1]$ with $s<t$,
\begin{equation}\label{eq:preliminaryHKestimate}
\tilde H(x,t;y,s) \leq \frac{C_1\exp\left(-\frac{d^2_{\bar g}(x,y)}{C_2(t-s)}\right)}{(t-s)^{n/2}}.
\end{equation}
Then, for $y\in M$ and $t>0$, we define $\tilde u(y,t) := \tilde H(p, 1; y, t)$. The proof of a universal bound of the form (\ref{eq:universalBackwardsHKbound}) in \cite[Theorem 2.3]{BamlerChen23} only uses the maximum principle (see, for instance, \cite[Theorem 12.14]{ChowEtAl08}) and the super-exponential spatial decay of the kernel as in (\ref{eq:preliminaryHKestimate}). Therefore, due to the uniform curvature bound for $\tilde g_t$ for $t\in [\tfrac{1}{2}, 1]$ in our setting, we may apply their argument to find
\begin{equation}\label{eq:BackwardsHKboundRF}
\tilde u(\cdot, t) \leq Ct^{-(1-2/n)c_0}
\end{equation}
for all $0 < t < 1/2$. Now define $u$ by $u(x,t) := \tilde u(\chi_t^{-1}(x),t)$, so that $u$ is the backwards heat kernel for the Ricci-DeTurck flow, as expected, and the bound (\ref{eq:universalBackwardsHKbound}) is inherited from (\ref{eq:BackwardsHKboundRF}).
\end{proof}

\begin{lemma}\label{lemma:nonstandardHK}
In the setting of Theorem \ref{thm:nonstandardHK}, on any interval of the form $[a, 1]$ the kernel $u$ will satisfy bounds of the form
\begin{align*}
u(x,t) & \leq \frac{C_a}{(1-t)^{n/2}}\exp\left(-\frac{|x-y|^2}{D_a(1-t)}\right)\\
|\nabla^k u(x,t)| &\leq \frac{C_{a,k}}{(1-t)^{(n+k)/2}}\exp\left(-\frac{|x-y|^2}{D_a(1-t)}\right)
\end{align*}
where $C_a = C_a(a, n, c, b_1, b_2, v_0, \kappa)$, $C_{a,k} = C_{a,k}(a, k, n, c, b_1, b_2, v_0, \kappa)$, and $D_a = D_a(a, n, c, b_1, b_2, \kappa)$.
\end{lemma}
\begin{proof}
Both of these statements are due to the fact that on such intervals, the flow will have a curvature bound of the form $c/a$. The pointwise bound for $u$ is due to \cite[Theorem 26.25]{ChowEtAl10}, since $(g_t)_{t\in (0,1]}$ is uniformly bilipschitz to $\bar g$ for all $t\in (0,1]$ (for some time-independent bilipschitz constant). The derivative estimates are due to \cite[Theorem 8.12.1]{Krylov96}, after adjusting $D_a$; see also \cite[Corollary 2.7]{PBG19}.
\end{proof}

\begin{lemma}\label{lemma:RBounduIntegration}
Let $(M^n, \bar g)$ be a smooth complete Riemannian manifold such that $\Ric(\bar g) \geq \kappa$ for some $\kappa\in \R$ and such that $\inf_{x\in M, 0 < r < 1}|B(x, r)|_{\bar g}/r^n > v_0$ for some $v_0 > 0$. Suppose that $(\bar g_t)_{t\in [0,1]}$ is a smooth Ricci flow on $M$ such that for some $b_1>0$ and for all $t\in [0,1]$, $\bar g_t$ is $(1+b_1)$-bilipschitz to $\bar g$. Suppose that $(g_t)_{t\in (0,1]}$ is a smooth solution to the Ricci-DeTurck flow on $M$ with respect to the background Ricci flow $\bar g_t$, such that for some $b_2>0$, $g_t$ is $(1+b_2)$-bilipschitz to $\bar g_t$, for all $t\in (0,1]$. Suppose that there exists a constant $c>0$ such that for all $t\in (0,1]$ and $k= 1,2$, $||\nabla^k(g_t)||_{C^0(M)}\leq c/t^{k/2}$, where $\nabla$ denotes the derivative with respect to $\bar g_t$. Let $p\in M$, and suppose that $u$ is the backwards heat kernel based at $(p,1)$ given by Theorem \ref{thm:nonstandardHK}. Then, for all $0 < t < 1$, we have
\begin{equation}
R(p,1) \geq \int R(x,t)u(x,t)dg_t(x).
\end{equation}
\end{lemma}
\begin{proof}
Let $\chi_t$, $\tilde g_t$, and $\tilde u$ be as in the proof of Theorem \ref{thm:nonstandardHK}. 

First note that, since $g_t$ is uniformly bilipschitz to $\bar g$ for any $t\in (0,1]$ and $\bar g$ is complete, $g_t$ is complete as well, and hence $\tilde g_t$ is complete for any $t\in (0,1]$. Then, for any fixed $t_0 >0$ and $x\in M$, an estimate of Shi \cite[Theorem 3.3.3]{Topping06} implies that, since $|\Rm(\tilde g_t)| \leq 2C(n)/t_0$ for all $t\in [t_0/2, t_0]$,
 \begin{equation}
 |\nabla^{\tilde g_{t_0}}\Rm^{\tilde g_{t_0}}|(x) \leq \left(\frac{C(n)}{t_0}\right)\sqrt{\frac{1}{t_0} + \frac{2C(n)}{t_0}} \leq \frac{C(n)}{t_0^{3/2}},
 \end{equation}
 where $C(n)$ is adjusted. In particular, $\nabla^{\tilde g_t}R(\tilde g_t)$ is uniformly bounded by $C(n)/t^{3/2}$ in any fixed time slice.

Then, for any $t\in (0, 1)$ we have
\begin{align*}
\int_{\R^n}R^{g^t}(x)u_t(x)dg_t(x) &= \int_{\R^n}\chi_t^*(R^{g^t}u_t)(x)d(\chi_t^*g_t)(x)
\\&=  \int_{\R^n}R^{\tilde g^t}(x)\tilde u_t(x)d\tilde g_t(x).
\end{align*}
Then
\begin{align*}
\frac{d}{dt}\int_{\R^n}R^{g^t}(x)u_t(x)dg_t(x) &= \frac{d}{dt}\int_{\R^n}R^{\tilde g^t}(x)\tilde u_t(x)d\tilde g_t(x)
\\&\geq \int_{\R^n}\left( \Delta_{\tilde g_t} R^{\tilde g^t}(x) + \frac{2}{n}(R^{\tilde g_t}(x))^2\right)\tilde u_t(x)d\tilde g_t(x)
\\& + \int_{\R^n}R^{\tilde g^t}(x)\left(-\Delta_{\tilde g_t} \tilde u_t (x) + (1-\tfrac{2}{n})R^{\tilde g_t}(x)\tilde u_t(x)\right)d\tilde g_t(x)
\\& + \int_{\R^n}R^{\tilde g^t}(x)\tilde u_t(x)(-R^{\tilde g_t}(x))d\tilde g_t(x)
\\& = \int_{\R^n}\Delta_{\tilde g_t} R^{\tilde g^t}(x)\tilde u_t(x) - R^{\tilde g^t}(x)\Delta_{\tilde g_t} \tilde u_t (x)d\tilde g_t(x)
\\&=0,
\end{align*}
where in the last step we may apply the Green's identity due to the super-exponential decay of $u$ given by Lemma \ref{lemma:nonstandardHK} and the fact that, in any fixed time slice, $\nabla R$ is uniformly bounded.
Thus, we have, for any $t\in (0,1)$,
\begin{align*}
R(p,1) &= \int_{\R^n} R_t(x)u_t(x)dg_t(x) + \int_t^{1}\frac{d}{ds}\int_{\R^n}R_s(x)u_s(x)dg_s(x) ds
\\& \geq \int_{\R^n} R_t(x)u_t(x)dg_t(x).
\end{align*}
\end{proof}

\section{Step 1: spatial lower scalar curvature bound}\label{sec:spatialLowerBound}
In this section we prove Theorem \ref{thm:SingSetSpatialLowerbound}. In this step we only use the fact that the scalar curvature under Ricci-DeTurck flow is a supersolution to the heat equation; that is, at this stage we do not use the positive source term in (\ref{eq:RDTFscalarev}). Note that very similar arguments appeared in \cite{PBG19} and \cite{PBG23}; here the arguments are slightly adjusted to obtain a lower bound that depends on $x$.
\begin{proof}[Proof of Theorem \ref{thm:SingSetSpatialLowerbound}] 
We determine the $\bar C, \bar D, \bar c$ in the course of the proof. First, let $C(n, c) = C_2$ and $D(n, c) = D_2$, where $C_2$ and $D_2$ are the constants from Theorem \ref{thm:standardRDTFHK} (with $\bar g = \bar g_t = \delta$ for all $t\geq 0$). Let $D' = D/2$ and, for $x\in \R^n$ and $0 < t \leq 1$, choose $C' > 0$ such that
\begin{equation}\label{eq:C'bound}
C' \leq  \int_{\R^n} \frac{C}{t^{n/2}}\exp\left(-\frac{|x-y|_{\delta}^2}{D't/2}\right)d_{\delta}y,
\end{equation} 
and observe that the right side of (\ref{eq:C'bound}) does not depend on $t$ or $x$, so $C'$ may be chosen independently of $x$ and $t$ as well. In particular, we have $C' = C'(n, c)$ and $D' = D'(n, c)$. Let $C''(n) = C'(n)(n + |\kappa_0|) > C'(n)$.

Fix a constant $0 < c_0 \leq 1$, which will also be determined in the course of the proof. We will choose $c_0$ so that it depends only on $\beta$. We will choose $\bar C$, $\bar D$, and $\bar c$ so that for all $0 < t\leq 1$ and $d_{\delta}(x, \Sing) \geq \bar ct^\eta$,
\begin{equation}\label{eq:negativesum}
-\frac{\bar C n}{2t}\exp\left(-\frac{d_{\delta}^2(x, \Sing)t^{2\beta}}{\bar D t}\right) + \sum_{i=1}^{\infty}\frac{C''(n)}{(t/2^{i-1})}\exp\left(-\frac{c_0^2d_{\delta}^2(x, \Sing)(t/2^i)^{2\beta}}{D't/2^i}\right) <0.
\end{equation}
To see why this is possible, let $\bar c = \sqrt{D'/c_0^2}$. First observe that if $0 < t\leq 1$ and $d_{\delta}(x, \Sing) \geq \bar c t^{\eta}$ then
\begin{equation*}
\exp\left(\frac{c_0^2d_{\delta}^2(x, \Sing)}{D't^{1-2\beta}}\right) \geq \exp\left(\frac{1}{t^{1-2\beta-2\eta}}\right) \geq e >2.
\end{equation*}
Also, $(2^{1-2\beta})^{i}-1 \geq i$ for all $i$ sufficiently large depending on $\beta$, say, for all $i\geq m(\beta)$. In particular, for $i\geq m(\beta)$ we have
\begin{equation*}
\exp\left(-\frac{c_0^2d_{\delta}^2(x, \Sing)}{D't^{1-2\beta}}((2^{1-2\beta})^i - 1) \right) = \left[\exp\left(-\frac{c_0^2d_{\delta}^2(x, \Sing)}{D't^{1-2\beta}} \right)\right]^{((2^{1-2\beta})^i - 1)} \leq \left(\frac{1}{e}\right)^i,
\end{equation*}
so
\begin{align*}
\sum_{i=1}^{\infty} 2^i\exp\left(-\frac{c_0^2d_{\delta}^2(x, \Sing)}{D't^{1-2\beta}}((2^{1-2\beta})^i - 1) \right) &= \sum_{i=1}^{m(\beta) - 1} 2^i\exp\left(-\frac{c_0^2d_{\delta}^2(x, \Sing)}{D't^{1-2\beta}}((2^{1-2\beta})^i - 1) \right) 
\\& \qquad \qquad + \sum_{i=m(\beta)}^{\infty}2^i\exp\left(-\frac{c_0^2d_{\delta}^2(x, \Sing)}{D't^{1-2\beta}}((2^{1-2\beta})^i - 1) \right)
\\&\leq  \sum_{i=1}^{m(\beta) - 1} 2^i + \sum_{i=m(\beta)}^{\infty} \left(\frac{2}{e}\right)^i
\\&\leq C(\beta).
\end{align*}

In particular, for all $k\in \mathbb{N}$,  we have
\begin{align*}
&-\frac{\bar Cn}{2t}\exp\left(-\frac{d_{\delta}^2(x, \Sing)t^{2\beta}}{\bar Dt}\right) + \sum_{i=1}^{k}\frac{C''(n)}{(t/2^{i-1})}\exp\left(-\frac{c_0^2d_{\delta}^2(x, \Sing)(t/2^i)^{2\beta}}{D't/2^i}\right)
\\&= -\frac{\bar C n}{2t}\exp\left(-\frac{d_{\delta}^2(x, \Sing)t^{2\beta}}{\bar D t}\right) 
\\& \qquad \qquad + \frac{C''(n)}{2t}\exp\left(-\frac{c_0^2d_{\delta}^2(x, \Sing)}{D't^{1-2\beta}}\right)\sum_{i=1}^{k}2^i\exp\left(-\frac{c_0^2d_{\delta}^2(x, \Sing)(t/2^i)^{2\beta}}{D't/2^i} + \frac{c_0^2d_{\delta}^2(x, \Sing)}{D't^{1-2\beta}}\right)
\\& = -\frac{\bar C n}{2t}\exp\left(-\frac{d_{\delta}^2(x, \Sing)t^{2\beta}}{\bar D t}\right) 
\\& \qquad \qquad + \frac{C''(n)}{2t}\exp\left(-\frac{c_0^2d_{\delta}^2(x, \Sing)}{D't^{1-2\beta}}\right)\sum_{i=1}^{k}2^i\exp\left(-\frac{c_0^2d_{\delta}^2(x, \Sing)}{D't^{1-2\beta}}((2^{1-2\beta})^i-1)\right)
\\&\leq -\frac{\bar C n}{2t}\exp\left(-\frac{d_{\delta}^2(x, \Sing)t^{2\beta}}{\bar Dt}\right) + \frac{C(\beta)C''(n)}{2t}\exp\left(-\frac{c_0^2d_{\delta}^2(x, \Sing)}{D't^{1-2\beta}}\right),
\end{align*}
which is negative provided that $\bar C > C(\beta)C''(n, c)$ and $\bar D > D'(n,c)/c_0^2$. This determines $\bar C$, and it also determines $\bar D$ and $\bar c$ once $c_0$ is chosen.

We now proceed to the rest of the proof. Let $\Phi(x,t;y,s)$ denote the usual heat kernel for the Ricci-DeTurck flow, as in Theorem \ref{thm:standardRDTFHK}. We argue by contradiction.

Suppose $R(x,t) < \kappa_0 -\frac{\bar C n}{2t}\exp\left(-\frac{d_{\delta}^2(x, \Sing)t^{2\beta}}{\bar Dt}\right)$, for some $0 < t \leq 1$ and $d_{\delta}(x, \Sing) \geq \bar ct^{\eta}$. In what follows, all balls are measured with respect to the Euclidean metric.

\begin{claim*}
There exists $x' \in B(x, c_0d_{\delta}(x, \Sing)(t/2)^{\beta})$ such that 
\begin{equation}\label{eq:contradictionBaseCase}
R(x', t/2) \leq \kappa_0 -\frac{\bar C n}{2t}\exp\left(-\frac{d_{\delta}^2(x, \Sing)t^{2\beta}}{\bar Dt}\right) + \frac{C''(n)}{t}\exp\left(-\frac{c_0^2d_{\delta}^2(x, \Sing)(t/2)^{2\beta}}{D't/2}\right).
\end{equation}
\end{claim*}
\begin{proof}[Proof of Claim]
Suppose not, i.e. suppose that on $B(x, c_0d_{\delta}(x, \Sing)t^{\beta})$ we have
\begin{equation}\label{eq:contradictorypositivesum}
R(\cdot,  t/2) - \kappa_0 > -\frac{\bar C n}{2t}\exp\left(-\frac{d_{\delta}^2(x, \Sing)t^{2\beta}}{\bar D t}\right) + \frac{C''(n)}{t}\exp\left(-\frac{c_0^2d_{\delta}^2(x, \Sing)(t/2)^{2\beta}}{D't/2}\right).
\end{equation}
Observe that, by (\ref{eq:negativesum}), the right side of (\ref{eq:contradictorypositivesum}) is negative, so by Theorem \ref{thm:standardRDTFHK}, 
\begin{align*}
& -\frac{\bar C n}{2t}\exp\left(-\frac{d_{\delta}^2(x, \Sing)t^{2\beta}}{\bar D t}\right) > R(x,t) - \kappa_0
\\& \geq \int_{B(x, c_0d_{\delta}(x, \Sing)(t/2)^\beta)}\Phi(x,t;y,t/2)[R(y,t/2) - \kappa_0]d_{t/2}y 
\\& \qquad \qquad + \int_{\R^n\setminus B(x, c_0d_{\delta}(x, \Sing)(t/2)^\beta)}\Phi(x,t;y,t/2)[R(y,t/2) - \kappa_0]d_{t/2}y
\\&\geq -\frac{\bar C n}{2t}\exp\left(-\frac{d_{\delta}^2(x, \Sing)t^{2\beta}}{\bar D t}\right) + \frac{C''(n)}{t}\exp\left(-\frac{c_0^2d_{\delta}^2(x, \Sing)(t/2)^{2\beta}}{D't/2}\right) 
\\& \qquad \qquad -\left[\frac{n}{t} + \kappa_0\right]\int_{\R^n\setminus B(x, c_0d_{\delta}(x, \Sing)(t/2)^{\beta})}\Phi(x,t;y,t/2)dy
\\& \geq -\frac{\bar C(n)}{2t}\exp\left(-\frac{d_{\delta}^2(x, \Sing)t^{2\beta}}{\bar D t}\right) + \frac{C''(n)}{t}\exp\left(-\frac{c_0^2d_{\delta}^2(x, \Sing)(t/2)^{2\beta}}{D't/2}\right) 
\\& \qquad \qquad -C'(n)\left[\frac{n}{t} + \kappa_0\right]\exp\left(-\frac{c_0^2d_{\delta}^2(x, \Sing)(t/2)^{2\beta}}{D't/2}\right)
\\& \geq -\frac{\bar C n}{2t}\exp\left(-\frac{d_{\delta}^2(x, \Sing)t^{2\beta}}{\bar Dt}\right),
\end{align*}
where $C'$, $C''$, and $D'$ are as described above.
To summarize, we have shown that
\begin{equation*}
 -\frac{\bar C n}{2t}\exp\left(-\frac{d_{\delta}^2(x, \Sing)t^{2\beta}}{\bar D t}\right) > -\frac{\bar C n}{2t}\exp\left(-\frac{d_{\delta}^2(x, \Sing)t^{2\beta}}{\bar Dt}\right).
\end{equation*}
This is a contradiction, whence follows the claim.
\end{proof}

\begin{claim*}
Suppose there is a point $x^k \in \R^n$ such that
\begin{equation*}
R(x^k, t/2^k) <  \kappa_0 -\frac{\bar Cn}{2t}\exp\left(-\frac{d_{\delta}^2(x, \Sing)t^{2\beta}}{\bar Dt}\right) + \sum_{i=1}^{k}\frac{C''(n)}{(t/2^{i-1})}\exp\left(-\frac{c_0^2d_{\delta}^2(x, \Sing)(t/2^i)^{2\beta}}{D't/2^i}\right).
\end{equation*}
Then there exists a point $x^{k+1}\in B(x^k, c_0d_{\delta}(x, \Sing)(t/2^{k+1})^{\beta})$ such that 
\begin{equation*}
R(x^{k+1}, t/2^{k+1})< \kappa_0 -\frac{\bar C n}{2t}\exp\left(-\frac{d_{\delta}^2(x, \Sing)t^{2\beta}}{\bar D t}\right) + \sum_{i=1}^{k+1}\frac{C''(n)}{(t/2^{i-1})}\exp\left(-\frac{c_0^2d_{\delta}^2(x, \Sing)(t/2^i)^{2\beta}}{D't/2^i}\right).
\end{equation*}
\end{claim*}
\begin{proof}[Proof of Claim]
This is as above. Suppose the claim is false, so that we have
\begin{align*}
-\frac{\bar C n}{2t}&\exp\left(-\frac{d_{\delta}^2(x, \Sing)t^{2\beta}}{\bar D t}\right) + \sum_{i=1}^{k}\frac{C''(n)}{(t/2^{i-1})}\exp\left(-\frac{c_0^2d_{\delta}^2(x, \Sing)(t/2^i)^{2\beta}}{D't/2^i}\right) > R(x^k, t/2^k) - \kappa_0
\\& \geq \int_{B(x^k, c_0d_{\delta}(x, \Sing)(t/2^{k+1})^\beta)}\Phi(x^k,t;y,t/2^{k+1})[R(y,t/2^{k+1}) - \kappa_0]dy 
\\& \qquad \qquad \qquad + \int_{\R^n\setminus B(x^k, c_0d_{\delta}(x, \Sing)(t/2^{k+1})^\beta)}\Phi(x^k,t;y,t/2^{k+1})[R(y,t/2^{k+1}) - \kappa_0]dy
\\& \geq -\frac{\bar C n}{2t}\exp\left(-\frac{d_{\delta}^2(x, \Sing)t^{2\beta}}{\bar D t}\right) + \sum_{i=1}^{k+1}\frac{C''(n)}{(t/2^{i-1})}\exp\left(-\frac{c_0^2d_{\delta}^2(x, \Sing)(t/2^i)^{2\beta}}{D't/2^i}\right)
\\& \qquad \qquad \qquad -C'(n)\left[\frac{n}{(t/2^{k})} + \kappa_0\right]\exp\left(-\frac{c_0^2d_{\delta}^2(x, \Sing)(t/2^{k+1})^{2\beta}}{D't/2^{k+1}}\right),
\end{align*}
a contradiction.
\end{proof}

\begin{claim*}
There exists a sequence of points $(x^k)_{k=0}^{\infty}\in \R^n$ such that 
\begin{enumerate}
\item $x^0 = x$,
\item $d(x^k, x^{k+1}) \leq c_0d_{\delta}(x, \Sing)(t/2^{k+1})^{\beta}$,
\item we have
\begin{equation*}
R(x^k, t/2^k) < \kappa_0 -\frac{\bar C n}{2t}\exp\left(-\frac{d_{\delta}^2(x, \Sing)t^{2\beta}}{\bar D t}\right) + \sum_{i=1}^{k}\frac{C''(n)}{(t/2^{i-1})}\exp\left(-\frac{c_0^2d_{\delta}^2(x, \Sing)(t/2^i)^{2\beta}}{D't/2^i}\right),
\end{equation*}
and
\item $x^k\xrightarrow[k\to\infty]{}x^{\infty}$ for some $x^{\infty}\notin \Sing$.
\end{enumerate}
\end{claim*}
\begin{proof}[Proof of Claim]
Construct the sequence $(x^k)_{k=0}^{\infty}$ inductively by the previous two claims, so that the first three items are satisfied. The last item is satisfied due to the fact that
\begin{equation*}
\sum_{k=1}^{\infty}c_0d_{\delta}(x,\Sing)\left(\frac{t}{2^{k+1}}\right)^{\beta}
\end{equation*}
is a convergent series, and the tail of a convergent series tends to zero, so $(x^k)_{k=0}^{\infty}$ is be Cauchy, and hence (by completeness of $(\R^n,\delta)$) converges to some $x^{\infty}\in \R^n$. To see that $x^{\infty} \notin \Sing$, note that
\begin{equation*}
d_{\delta}(x^{\infty}, \Sing) \geq d_{\delta}(x, \Sing) - \sum_{k=1}^{\infty} c_0d_{\delta}(x, \Sing)(t/2^{k})^{\beta} = (1 - c_{\beta}c_0t^{\beta})d_{\delta}(x, \Sing) >0,
\end{equation*}
where
\begin{equation*}
c_\beta = \sum_{k=1}^{\infty} \left(\frac{1}{2^{\beta}}\right)^k,
\end{equation*}
provided that $c_0$ is chosen sufficiently smaller than $1$ so that (since $t\leq 1$) $c_{\beta}c_0t^{\beta} < 1$. Therefore, fix $0 < c_0 < 1$ so that this holds. This determines $\bar D = \bar D(n,\beta)$ and $\bar c = \bar c(n, \beta)$.
\end{proof}

Having proven this claim, note that
\begin{align*}
d(x^k, x^{\infty}) \leq \sum_{i=k+1}^{\infty}c_0d_{\delta}(x, \Sing)(t/2^{i})^{\beta} = c_0d_{\delta}(x, \Sing)(t/2^k)^{\beta}\sum_{i=1}^{\infty}(t/2^i)^\beta \leq c_\beta c_0d_{\delta}(x, \Sing)(t/2^k)^{\beta}.
\end{align*}

In particular,  by (\ref{eq:negativesum}) we have 
\begin{equation}\label{eq:negativescalarcurvlimit}
\begin{split}
\liminf_{t\searrow 0}\inf_{B(x^{\infty}, c_{\beta}c_0d_{\delta}(x, \Sing) (t/2^k)^{\beta})} R^{g_t} \leq -\frac{\bar C n}{2t}&\exp\left(-\frac{d_{\delta}^2(x, \Sing)t^{2\beta}}{\bar D t}\right) 
\\& + \sum_{i=1}^{\infty}\frac{C''(n)}{(t/2^{i-1})}\exp\left(-\frac{c_0^2d_{\delta}^2(x, \Sing)(t/2^i)^{2\beta}}{D't/2^i}\right) <\kappa_0,
\end{split}
\end{equation}
which contradicts the $\beta$-weak (\ref{eq:betaweak}) condition at $x^{\infty}\notin \Sing$ with the lower bound $\kappa_0$.
\end{proof}

\section{Step 2: improved lower bound due to source term}\label{sec:uniformLowerBound}
In this section we use Theorem \ref{thm:nonstandardHK} to improve the lower bound prove in Section \ref{sec:spatialLowerBound}. We first prove Theorem \ref{thm:discretesetsmoothing}:

\begin{proof}[Proof of Theorem \ref{thm:discretesetsmoothing}]
It is sufficient to assume that $(g_t)_{t\in (0,T)}$ is defined for some $T>1$ and to show that $R(g_1)\geq 0$, since otherwise the result for $t>0$ follows from parabolic rescaling. By (\ref{item:discretesetconvergence}) in the statement of the theorem, $(g_t)_{t\in (0,T)}$ satisfies the condition (\ref{eq:betaweak}) with the lower bound $0$ for any $\beta > 0$. Note that since $0 < c_0 < n/2$, we have
\begin{equation*}
c_0 < n/2 = \frac{n/2(1 - 2/n)}{1 - 2/n} = \frac{n/2 - 1}{1 - 2/n},
\end{equation*}
so
\begin{equation*}
\frac{1 + (1-2/n)c_0}{n} < 1/2
\end{equation*}
and hence there exists some $\beta >0$ such that 
\begin{equation}\label{eq:expinequalitydiscrete}
\frac{1 + (1-2/n)c_0}{n} < \frac{1 - 2\beta}{2}.
\end{equation}
Let $\eta \in ([1 + (1-2/n)c_0]/n, (1-2\beta)/2)$ so that we may apply Theorem \ref{thm:SingSetSpatialLowerbound}. Then, for sufficiently small $t\in (0,1)$, we have 
\begin{align*}
R(p,1)&\geq \int_{\R^n}R(t)u(t)dg_t
\\& \geq \int_{\T(\Sing, ct^\eta)} R(t)u_tdg_t + \int_{\R^n\setminus \T(\Sing, ct^\eta)} R(t)u_tdg_t 
\\& \geq -\frac{c_0}{t}Ct^{-(1-2/n)c_0}|\T(\Sing, ct^\eta)| + \int_{\R^n\setminus \T(\Sing, ct^\eta)} -\frac{\bar C}{t}\exp\left(-\frac{d_\delta^2(x, \Sing)}{\bar Dt^{1-2\beta}}\right)ct^{-(1-2/n)c_0}dg_t 
\\& \geq -Ct^{-1}t^{-(1-2/n)c_0}t^{n\eta} + \int_{\R^n\setminus \T(\Sing, ct^\eta)} -\frac{\bar C}{t}\exp\left(-\frac{d_\delta^2(x, \Sing)}{\bar Dt^{1-2\beta}}\right)ct^{-(1-2/n)c_0}dg_t 
\\& \geq I_t + II_t,
\end{align*}
for some $I_t$ and $II_t$, where Lemma \ref{lemma:RBounduIntegration} is used in the first step and Theorem \ref{thm:nonstandardHK} and Theorem \ref{thm:SingSetSpatialLowerbound} are used in the third step. 

We now show that we can pick the lower bounds $I_t$ and $II_t$ so that $I_t, II_t \xrightarrow[t\searrow 0]{} 0$. First of all, we have 
\begin{equation*}
-1 + n\eta -(1-2/n)c_0 > -1 + 1 + (1-2/n)c_0 - (1-2/n)c_0 = 0,
\end{equation*}
so
\begin{equation*}
I_t := -\frac{c}{t}t^{\eta n}t^{-(1 - 2/n)c_0} \xrightarrow[t\searrow 0]{} 0.
\end{equation*}

To see why there is a lower bound $II_t$ such that $II_t\xrightarrow[t\searrow 0]{} 0$, first observe that by Lemma \ref{lemma:CCSingularSet}, $\Sing$ has finite diameter. Fix some point $p\in \Sing$.  If $x\in \R^n$ such that $d(x,p)\geq 2\diam(\Sing)$, then $d(x,p) - \diam(\Sing)\geq \tfrac{1}{2}d(x,p)$. In particular, by the triangle inequality we have
\begin{equation*}
d(x,\Sing) \geq d(x,p) - \diam(\Sing) \geq \tfrac{1}{2}d(x,p),
\end{equation*}
so, adjusting the constants $C$ and $D$, we have
\begin{align*}
& \int_{\R^n\setminus B(p, 2\diam(\Sing))} \frac{C}{t}\exp\left(-\frac{d^2(x, \Sing)}{Dt^{1-2\beta}}\right)ct^{-(1-2/n)c_0}dg_t 
\\& \qquad \qquad + \int_{B(p, 2\diam(\Sing))\setminus \T(\Sing, ct^\eta)} \frac{C}{t}\exp\left(-\frac{d^2(x, \Sing)}{Dt^{1-2\beta}}\right)ct^{-(1-2/n)c_0}dg_t
\\& \leq \int_{\R^n\setminus B(p, 2\diam(\Sing))} \frac{C}{t}\exp\left(-\frac{d^2(x, p)}{Dt^{1-2\beta}}\right)ct^{-(1-2/n)c_0}dg_t 
\\& \qquad \qquad +  \int_{B(p, 2\diam(\Sing))\setminus \T(\Sing, ct^\eta)} \frac{C}{t}\exp\left(-\frac{d^2(x, \Sing)}{Dt^{1-2\beta}}\right)ct^{-(1-2/n)c_0}dg_t
\\& \leq Ct^{-1}t^{-(1-2/n)c_0}t^{\tfrac{n}{2}(1-2\beta)}\int_{\R^n}\exp\left(-\frac{|x|^2}{D}\right)d\delta
\\& \qquad \qquad + Ct^{-1}t^{-(1-2/n)c_0}\exp\left(-\frac{1}{t^{1-2\beta-2\eta}}\right)|B(p, 2\diam(\Sing))|
\\& =: -II_t.
\end{align*}
Since $\eta \in ([1 + (1-2/n)c_0]/n, (1-2\beta)/2)$ we have $1 - 2\beta - 2\eta >0$, and by (\ref{eq:expinequalitydiscrete}) we have $-1 -(1-2/n)c_0 + (n/2)(1-2\beta) > 0$ so $II_t \xrightarrow[t\searrow 0]{}0$.
\end{proof}

We now prove Theorem \ref{thm:betaNNSCLinftySet}, arguing in a similar manner to the proof of Theorem \ref{thm:discretesetsmoothing}
\begin{proof}[Proof of Theorem \ref{thm:betaNNSCLinftySet}]
As above, it is sufficient to show $R(g_1)\geq 0$, since the result for other $t>0$ follows from parabolic rescaling. Choose $\bar \beta$ sufficiently small depending on $\alpha$ so that
\begin{equation*}
\alpha\left(\frac{1-2\bar\beta}{2}\right) >1;
\end{equation*}
this is possible because $\alpha >2$. Let $c(n)>0$ be  a constant depending only on $n$ with the property that if, for $k = 0,1,2$ and all $t>0$, $||\nabla^k (g_t - \delta)||_{L^\infty(\R^n)} \leq \bar \varepsilon/t^{k/2}$, then $|R(g_t)|_x| \leq c(n)\bar \varepsilon/ t$ for all $x\in \R^n$. Pick $\bar \varepsilon > 0$ depending on $n$ so that $\bar \varepsilon < n/(2c(n))$. Then, since $n\geq 3$,
\begin{equation*}
1 - \frac{2}{n}\left[1 + (1-\tfrac{2}{n})c(n)\bar\varepsilon \right] > 1 - \frac{2}{n}\left[1 + (1-\tfrac{2}{n})c(n)\tfrac{n}{2c(n)} \right] = 0,
\end{equation*}
so we may reduce $\bar \beta >0$ as necessary depending on $\bar \varepsilon = \bar \varepsilon(n)$ so that
\begin{equation*}
\bar \beta \in \left( 0,  \frac{1}{2}\left[1 - \frac{2}{n}\left[1 + (1-\tfrac{2}{n})c(n)\bar\varepsilon \right]\right]  \right).
\end{equation*}
 Then
\begin{equation}\label{eq:exponentinequality}
\frac{1-2\bar\beta}{2} > \frac{1 + (1-\tfrac{2}{n})c(n)\bar\varepsilon}{n},
\end{equation}
and observe that for all $\varepsilon < \bar \varepsilon$ we also have
\begin{equation*}
\frac{1-2\bar\beta}{2} > \frac{1 + (1-\tfrac{2}{n})c(n)\varepsilon}{n}.
\end{equation*}
In particular, we may reduce $\bar \varepsilon$ as needed and still preserve (\ref{eq:exponentinequality}). Reduce $\bar\varepsilon$ further as needed depending on $n$ and $\alpha$ so that
\begin{equation}\label{eq:expinequalitysingsetsize}
\alpha\left(\frac{1-2\bar\beta}{2}\right) - 1 > (1 - \tfrac{2}{n})c(n)\bar \varepsilon.
\end{equation}

Now let $c_0 = c(n)\bar\varepsilon$, let $0 < \beta < \bar \beta$, let 
\begin{equation}\label{eq:etainterval}
\frac{1-2\bar\beta}{2} < \eta < \frac{1-2\beta}{2},
\end{equation} 
and let $\varepsilon < \bar \varepsilon$. By definition of $c_0$ and $c(n)$, $R(g_t) \geq -c_0/t$ and as above we use Theorem \ref{thm:nonstandardHK}, Lemma \ref{lemma:RBounduIntegration}, and Theorem \ref{thm:SingSetSpatialLowerbound} to find that for all sufficiently small $t>0$,
\begin{align*}
R(p,1)&\geq \int_{\R^n}R(t)u(t)dg_t
\\& \geq \int_{\T(\Sing, ct^\eta)} R(t)u_tdg_t + \int_{\R^n\setminus \T(\Sing, ct^\eta)} R(t)u_tdg_t 
\\& \geq -\frac{c_0}{t}Ct^{-(1-2/n)c_0}|\T(\Sing, ct^\eta)| + \int_{\R^n\setminus \T(\Sing, ct^\eta)} -\frac{\bar C}{t}\exp\left(-\frac{d_\delta^2(x, \Sing)}{\bar Dt^{1-2\beta}}\right)ct^{-(1-2/n)c_0}dg_t 
\\& \geq -Ct^{-1}t^{-(1-2/n)c_0}t^{\alpha\eta} + \int_{\R^n\setminus \T(\Sing, ct^\eta)} -\frac{\bar C}{t}\exp\left(-\frac{d_\delta^2(x, \Sing)}{\bar Dt^{1-2\beta}}\right)ct^{-(1-2/n)c_0}dg_t 
\\& \geq I_t + II_t,
\end{align*}
for some $I_t$ and $II_t$.

As before we show that we can pick the lower bounds $I_t$ and $II_t$ so that $I_t, II_t \xrightarrow[t\searrow 0]{} 0$. First, to see why $I_t\xrightarrow[t\searrow 0]{} 0$, note that by (\ref{eq:etainterval}) and (\ref{eq:expinequalitysingsetsize}) we have
\begin{align*}
\alpha\eta - (1-\frac{2}{n})c_0 - 1 & > \alpha\left(\frac{1-2\bar\beta}{2}\right) - (1-\frac{2}{n})c_0 - 1
\\&> (1 - \frac{2}{n})c(n)\bar\varepsilon - (1-\frac{2}{n})c_0 = 0,
\end{align*}
that so we may choose $I_t := -Ct^{\omega}$ for some $\omega > 0$ so that $I_t\xrightarrow[t\searrow 0]{} 0$.

As before, fix some $p\in \Sing$, and estimate as in the proof of Theorem \ref{thm:discretesetsmoothing}, adjusting $C$ and $D$, to find
\begin{align*}
\int_{\R^n\setminus \T(\Sing, ct^\eta)} & -\frac{\bar C}{t}\exp\left(-\frac{d_\delta^2(x, \Sing)}{\bar Dt^{1-2\beta}}\right)ct^{-(1-2/n)c_0}dg_t 
\\& \geq - Ct^{-1}t^{-(1-2/n)c_0}t^{\tfrac{n}{2}(1-2\beta)}\int_{\R^n}\exp\left(-\frac{|x|^2}{D}\right)d\delta
\\& \qquad \qquad - Ct^{-1}t^{-(1-2/n)c_0}\exp\left(-\frac{1}{t^{1-2\beta-2\eta}}\right)|B(p, 2\diam(\Sing))|.
\end{align*}

By (\ref{eq:etainterval}), 
\begin{equation*}
Ct^{-1}t^{-(1-2/n)c_0}\exp\left(-\frac{1}{t^{1-2\beta-2\eta}}\right)|B(p, 2\diam(\Sing))| \xrightarrow[t\searrow 0]{}0,
\end{equation*}
and by (\ref{eq:exponentinequality}), 
\begin{equation*}
Ct^{-1}t^{-(1-2/n)c_0}t^{\tfrac{n}{2}(1-2\beta)}\int_{\R^n}\exp\left(-\frac{|x|^2}{D}\right)d\delta \xrightarrow[t\searrow]{}0.
\end{equation*}
Therefore, we have
\begin{align*}
& \int_{\R^n\setminus \T(\Sing, ct^\eta)} -\frac{C}{t}\exp\left(-\frac{d^2(x, \Sing)}{Dt^{1-2\beta}}\right)ct^{-(1-2/n)c_0}dg_t
\\& \geq -Ct^{-1}t^{-(1-2/n)c_0}\exp\left(-\frac{1}{t^{1-2\beta-2\eta}}\right)|B(p, 2\diam(\Sing))| 
\\& \qquad \qquad - Ct^{-1}t^{-(1-2/n)c_0}t^{\tfrac{n}{2}(1-2\beta)}\int_{\R^n}\exp\left(-\frac{|x|^2}{D}\right)d\delta =: II_t \xrightarrow[t\searrow]{}0.
\end{align*}
This completes the proof.
\end{proof}
\begin{remark}
In contrast to the proof of Theorem {\ref{thm:discretesetsmoothing}}, one should be able to carry out the proof of Theorem {bestNNSCLinftySet} with $u$ replaced by the backwards heat kernel for the standard heat equation for the Ricci-DeTurck flow after reducing $\bar \varepsilon$, though of course the proof would then yield a smaller threshold $\bar \varepsilon$.
\end{remark}

Theorem \ref{thm:smoothNNSCLinftySet} now follows quickly:
\begin{proof}[Proof of Theorem \ref{thm:smoothNNSCLinftySet}]
Let $\bar \varepsilon$ be as in Theorem \ref{thm:betaNNSCLinftySet}, and reduce $\bar \varepsilon$ as needed so that it is smaller than the constant $\varepsilon$ in Theorem \ref{thm:KL+}. Then, if $|| g - \delta||_{L^\infty(\R^n)} < \bar \varepsilon$, there exists a smooth Ricci-DeTurck flow $(g_t)_{t>0}$ starting from $g$ in the sense of Theorem \ref{thm:KL+}, so that $g_t \xrightarrow[t\searrow 0]{C^\infty_{loc}(\R^n\setminus \Sing)} g$ and for $k = 0, 1, 2, \ldots$,
\begin{equation*}
|| \nabla^k (g_t - \delta)||_{C^0(\R^n)}\leq \frac{c(k)\bar \varepsilon}{t^{k/2}}.
\end{equation*}
Reduce $\bar \varepsilon$ further as needed so that $c(0)\bar\varepsilon$, $c(1)\bar\varepsilon$, and $c(2)\bar \varepsilon$ are smaller than the constant $\bar \varepsilon(\alpha, n)$ from Theorem \ref{thm:betaNNSCLinftySet}.

Since $R(g)|_x\geq 0$ for $x\notin \Sing$ and since 
\begin{equation*}
g_t \xrightarrow[t\searrow 0]{C^\infty_{loc}(\R^n\setminus \Sing)} g,
\end{equation*} 
$(g_t)_{t>0}$ satisfies the $\beta$-weak condition (\ref{eq:betaweak}) at $x$ for any $0 < \beta < 1/2$ and all $x\in \R^n\setminus \Sing$. Let $\beta < \bar \beta$, where $\bar \beta$ is the constant given by Theorem \ref{thm:betaNNSCLinftySet}. Then, by Theorem \ref{thm:betaNNSCLinftySet}, $R(g_t)\geq 0$ for all $t>0$. This completes the proof.
\end{proof}

\appendix
\section{Derivative estimates for Ricci-DeTurck flow from initial data with higher regularity}\label{appendix:convergence}

In addition to the norms $|| \cdot||_{X_T}$ introduced in Theorem \ref{theorem:KL}, in this section we will work with the following norms:
\begin{equation}
|| f||_{Y_T^0} = \sup_{x\in \R^n}\sup_{0 < r < \sqrt{T}}\left( r^{-n}|| f||_{L^1(B(x,r)\times (0, r^2))} + r^{4/(n+4)}|| f||_{L^{(n+4)/2}(B(x,r)\times (r^2/2, r^2))}\right)
\end{equation}
and
\begin{equation}
||f||_{Y_T^1} = \sup_{x\in \R^n}\sup_{0 < r < \sqrt{T}}\left(r^{-n/2}|| f||_{L^2(B(x,r)\times (0, r^2))} + r^{2/(n+4)}|| f||_{L^{n+4}(B(x,r)\times (r^2/2, r^2))}\right).
\end{equation}

We first show the following (cf. \cite[Lemma 4.2]{Appleton18}, \cite[Lemma 4.2]{Shi89}):
\begin{lemma}\label{lemma:smoothRDTFderivbounds}
Let $\varepsilon$ be as in Theorem \ref{theorem:KL}. There exists some positive $\varepsilon' = \varepsilon'(n) \leq \varepsilon$ such that the following is true: 

For some $k\geq 1$, let $g$ be a $C^k$ Riemannian metric on $\R^n$ such that $g$ is $(1+\varepsilon')$-bilipschitz to $\delta$.  For $t>0$ let $ g_t:= h_t + \delta$, where $(h_t)_{t>0}$ is the solution to the integral equation (\ref{eq:hevolution}) with initial data $g - \delta$ given by Theorem \ref{theorem:KL}. Suppose there exists $N >0$ such that $|| g - \delta||_{C^k(\R^n)} \leq N$. Then there exists $M = M(n, N) >0$ and $T = T(n, N) >0$ such that $|| g_t - \delta ||_{C^{k}(\R^n)} \leq M$ for all $t\in [0, T]$.
\end{lemma}
\begin{proof}
Let $\varepsilon' = \varepsilon(n)$ and $C = C(n)$ be as in Theorem \ref{theorem:KL}. Reduce $\varepsilon'$ depending on $C$ (and hence on $n$) so that 
\begin{equation}\label{eq:Cepsilonsmall}
C\varepsilon' < 1/8.
\end{equation}
We will reduce $\varepsilon'$ further as necessary (depending on $n$, but not $k$) in the course of this proof. Let $g$ be a $C^1$ Riemannian metric on $\R^n$ satisfying the hypotheses of the lemma.  Since $g$ is $(1+\varepsilon')$-bilipschitz to $\delta$, by Theorem \ref{theorem:KL} exists a smooth solution $(h_t)_{t>0}$ to (\ref{eq:Integraleq}) with initial data $g-\delta$ such that $||h_t||_{X_\infty} \leq C\varepsilon'$. For $t>0$ let $g_t := h_t + \delta$. 

For $m, j\in \mathbb{N}$ let $P(m, j):= \{(i_1, \ldots, i_j)\in \mathbb{N}^j: i_1 + \cdots + i_j = m\}$. Observe that, for all $k\in \mathbb{N}$
\begin{equation}\label{eq:higherorderQ1}
\nabla^k\left( (g^{-1} - \delta^{-1})\star \nabla g\right) \big|_t = (g^{-1} - \delta^{-1})\star \nabla^{k+1}g + \sum_{j=2}^{k+1} (g^{-1})^j\star \left(\sum_{(i_1,\ldots, i_j)\in P(k+1, j)} \nabla^{i_1} g \star \cdots \star \nabla^{i_j} g\right) \bigg|_t
\end{equation}
and
\begin{equation}\label{eq:higherorderQ0}
\nabla^k\left( g^{-1}\star g^{-1}\star \nabla g \star \nabla g\right) \big|_t
= \sum_{j=2}^{k+2}(g^{-1})^j\star\left(\sum_{(i_1,\ldots, i_j)\in P(k+2, j)} \nabla^{i_1} g \star \cdots \star \nabla^{i_j}g\right)\bigg|_t
\end{equation}
where, for two tensors $A$ and $B$, we write $A \star B$ to denote a linear combination of products of the coefficients of $A$ and $B$, we write $A^j$ to denote $A\star\cdots \star A$, where the $\star$ operation is performed $j-1$ times.
 
 We prove the result by induction on $k$. We first show the result for $k = 1$. Fix some $T>0$, which we will specify in the course of the proof. By (\ref{eq:Integraleq}), (\ref{eq:higherorderQ1}), and (\ref{eq:higherorderQ0}) we have
\begin{equation}
\begin{split}
\nabla g_t (x) &= \int_{\R^n} \bar K(x,t;y,0)\nabla g(y)dy 
\\& + \int_0^t \int_{\R^n} \bar K(x,t;y,s)\left[\nabla^2g_s\star \nabla g_s + \nabla  g_s* \nabla g_s * \nabla g_s + \nabla^*( h_s * \nabla^2 g_s + \nabla g_s * \nabla g_s)\right]_y dyds,
\end{split}
\end{equation}
so
\begin{align*}
|| \nabla g_t||_{X_T} & \leq \left| \left|  \int_{\R^n} \Phi(x,t;y,0)\nabla g(y)dy  \right| \right|_{X_T}
\\& + \left| \left| \int_0^t \int_{\R^n} \Phi(x,t;y,s)\left[ \nabla^2 g_s * \nabla g_s\right]_y dyds \right| \right|_{X_T}
\\& +  \left| \left| \int_0^t \int_{\R^n} \Phi(x,t;y,s)\left[ \nabla g_s * \nabla g_s * \nabla g_s\right]_y dyds \right| \right|_{X_T}
\\& + \left| \left| \int_0^t \int_{\R^n} \Phi(x,t;y,s)\left[\nabla^*(\nabla g_s * \nabla g_s)\right]_y dyds \right| \right|_{X_T} 
\\& + \left| \left| \int_0^t \int_{\R^n} \Phi(x,t;y,s)\left[\nabla^*( (g_s - \delta) * \nabla^2 g_s)\right]_y dyds \right| \right|_{X_T}
\\& =: A + B_1 + B_2 + C_1 + C_2. 
\end{align*}

By \cite[Lemma 2.2]{KochLamm12}, we have
\begin{equation}
A \leq c(n)|| \nabla g||_{L^\infty(\R^n)}.
\end{equation}
By \cite[Lemma 4.2]{KochLamm12} and the definitions of $|| \cdot||_{Y_T^0}$, $|| \cdot||_{Y_T^1}$, and $|| \cdot||_{X_T}$ (cf. \cite[Lemma 4.1]{KochLamm12}) we have
\begin{equation}
B_1 \leq c ||  \nabla^2 g_t * \nabla g_t||_{Y^0_T} \leq c || \nabla g_t||_{X_T}||h_t||_{X_T} \leq cC\varepsilon'  || \nabla g_t||_{X_T},
\end{equation}
where $c = c(n)$ is the constant from \cite[Lemma 4.2]{KochLamm12}.
Similarly,
\begin{equation}
B_2 \leq c || \nabla g_t * \nabla g_t * \nabla g_t||_{Y^0_T} \leq c|| \nabla g_t||_{L^\infty(\R^n\times [0,T])}||h_t||_{X_t}|| h_t||_{X_T} \leq cC(\varepsilon ')^2|| \nabla g_t||_{X_T}.
\end{equation}
Similarly, we have
\begin{equation}
C_1 \leq c|| \nabla g_t * \nabla g_t ||_{Y^1_T} \leq c || \nabla g_t||_{L^\infty(\R^n\times [0,T])}|| \nabla g_t||_{Y^1_T} \leq c|| \nabla g_t||_{X_T}|| h_t||_{X_T} \leq cC\varepsilon' ||\nabla g_t||_{X_T}
\end{equation}
and
\begin{equation}
C_2  \leq c|| (g_t - \delta) * \nabla^2 g_t ||_{Y^1_T} \leq c|| g_t - \delta||_{L^\infty(\R^n\times [0,T])}|| \nabla^2 g_t ||_{Y^1_T} \leq cC\varepsilon' || \nabla g_t ||_{X_T}.
\end{equation}
Combining the previous equations, we find that
\begin{equation}
|| \nabla g_t ||_{X_T} \leq c || \nabla g||_{L^\infty(\R^n)} + cC(3\varepsilon' + (\varepsilon')^2) ||\nabla g_t||_{X_T}. 
\end{equation}
Therefore, reducing $\varepsilon'$ (depending on $n$) so that $cC(3\varepsilon' + (\varepsilon')^2) < 1$, the previous equation implies that
\begin{equation}
|| \nabla g_t||_{L^\infty(\R^n\times [0, T])} \leq || \nabla g_t||_{X_T} \leq \frac{c}{1 - cC(3\varepsilon' + (\varepsilon')^2)}|| \nabla g||_{C^0(\R^n)}.
\end{equation}
This proves the result for $k=1$, with $M = \max\{cN/(1 - cC(3\varepsilon' + (\varepsilon')^2)) , C\varepsilon'\}$.

Now suppose the statement of the lemma holds for $\varepsilon'$ as determined above, for some $k\geq 1$. Suppose that $g$ is a $C^{k+1}$ Riemannian metric on $\R^n$ such that $g$ is $(1+\varepsilon')$-bilipschitz to $\delta$ on $\R^n$, such that $|| g - \delta||_{C^{k+1}(\R^n)} \leq N$. Then $|| g - \delta||_{C^k(\R^n)} \leq N$, so by the inductive hypothesis there exists $M_k = M_k(n, N)$ and $T_k = T_k(n, N)$ such that $|| g_t - \delta ||_{C^k(\R^n\times [0,T_k])} \leq M_k$. We show that there exist $M_{k+1} = M_{k+1}(n, N)$ and $T_{k+1} = T_{k+1}(n, N)$ such that $|| g_t - \delta ||_{C^{k+1}(\R^n\times [0, T_{k+1}])} \leq M_{k+1}$. To see why, let $T \leq T_k$ and note that by (\ref{eq:higherorderQ0}), (\ref{eq:higherorderQ1}), (\ref{eq:hevolution}), and \cite[Lemma 4.2]{KochLamm12},  we have
\begin{align*}
|| \nabla^{k+1} g_t||_{X_T} & \leq || \nabla^{k+1} g_0||_{L^\infty(\R^n)} 
\\& + || \nabla^{k+2}g_t * \nabla g_t ||_{Y^0_T} + || \nabla^{k+1} g_t * \nabla ^2 g_t ||_{Y^0_T} + \sum_{\ell = 3}^{k} || \nabla^{\ell}* \nabla^{k+3-\ell} ||_{Y^0_T}
\\& + ||\nabla^{k+1} g_t * \nabla g_t * \nabla g_t ||_{Y^0_T} + \sum_{\ell = 2}^{k}\sum_{m = 1}^{k-\ell} || \nabla^{\ell}g_t * \nabla^{m}g_t * \nabla^{k+3-\ell - m} ||_{Y^0_T}
\\& + \sum_{j=4}^{k+3}\left(  \sum_{(i_1,\ldots, i_j)\in P(k+3, j)} \left| \left| \nabla^{i_1} g_t * \cdots * \nabla^{i_j}g_t  \right| \right|_{Y^0_T} \right)
\\& +   || h_t* \nabla^{k+2}g_t||_{Y^1_T} + || \nabla^{k+1} g_t * \nabla^{1}g_t ||_{Y^1_T} + \sum_{\ell = 2}^{k} ||  \nabla^\ell g_t * \nabla^{k+2-\ell}||_{Y^1_T}
\\& + \sum_{j=3}^{k+2} \left(\sum_{(i_1,\ldots, i_j)\in P(k+2, j)} | |\nabla^{i_1} g_t * \cdots * \nabla^{i_j} g_t | |_{Y^1_T}\right) 
\\& =: || \nabla^{k+1} g_0||_{L^\infty(\R^n)} 
\\& + A_1 + A_2 + A_3
\\& + A_4 + A_5
\\& + A_6
\\& + B_1 + B_2 + B_3
\\& + B_4.
\end{align*}

We now estimate each term separately. By the definitions of $|| \cdot||_{Y_T^0}$ and $|| \cdot||_{Y_T^1}$ and the inductive hypothesis, we have
\begin{equation}
A_1 \leq  || \nabla g_t ||_{Y_T^0}|| \nabla^{k+2} g_t||_{Y_T^0} \leq || h_t||_{X_T}|| \nabla^{k+1} g_t||_{X_T} \leq C\varepsilon' || \nabla^{k+1} g_t||_{X_T},
\end{equation}
\begin{equation}
\begin{split}
A_2 & \leq || \nabla^2 g_t * \nabla^{k+1}g_t ||_{Y^0_T} \leq || \nabla^2 g_t||_{L^\infty(\R^n\times [0,T])}|| \nabla^{k+1}g_t ||_{Y^0_T} 
\\& \leq M_k ||\nabla^{k+1}g_t||_{L^\infty(\R^n\times [0,T])} T \leq M_kT||\nabla^{k+1} g_t||_{X_T},
\end{split}
\end{equation}
\begin{equation}
A_3 \leq \sum_{\ell =3}^{k} || \nabla^\ell g_t||_{L^\infty(\R^n\times [0,T])} || \nabla^{k+3-\ell} g_t||_{L^\infty(\R^n\times [0,T])}|| 1||_{Y^0_T} \leq (k-2)M_k^2 T,
\end{equation}
\begin{equation}
A_4 \leq || \nabla g_t||^2_{C^0(\R^n\times[0,T])}|| \nabla^{k+1} g_t||_{Y^0_T} \leq M_k^2 ||\nabla^{k+1}g_t||_{C^0(\R^n\times[0,T])}|| 1||_{Y^0_T} \leq M_k^2T|| \nabla^{k+1}g_t||_{X_T}
\end{equation}
\begin{equation}
\begin{split}
A_5 &\leq \sum_{\ell = 2}^{k}\sum_{m = 1}^{k-\ell} || \nabla^{\ell}g_t||_{C^0(\R^n\times [0,T])} || \nabla^{m}g_t||_{C^0(\R^n\times [0,T])} ||\nabla^{k+3-\ell - m} ||_{C^0(\R^n\times[0,T])}|| 1||_{Y^0_T} 
\\& \leq (k-1)^2M_k^3T
\end{split}
\end{equation}
and
\begin{equation}
\begin{split}
A_6 & \leq \sum_{j=4}^{k+3}\left(  \sum_{(i_1,\ldots, i_j)\in P(k+3, j)} || \nabla^{i_1} g_t||_{C^0(\R^n\times[0,T])} \cdots ||\nabla^{i_j}g_t||_{C^0(\R^n\times[0,T])} || 1||_{Y^0_T} \right)
\\& \leq c(k)M_k^{k+3}T.
\end{split}
\end{equation}
Similarly,
\begin{equation}
B_1 \leq || g_t - \delta||_{L^\infty(\R^n\times [0,T])} || \nabla^{k+2} g_t||_{Y_T^1} \leq || g_t - \delta||_{X_T}|| \nabla^{k+1}g_t||_{X_T} \leq C\varepsilon ||\nabla^{k+1}g_t||_{X_T},
\end{equation}
\begin{equation}
B_2 \leq || \nabla g_t||_{L^\infty(\R^n\times [0,T])}|| \nabla^{k+1} g_t||_{L^\infty(\R^n\times [0,T])}\sqrt{T} \leq M_k\sqrt{T}|| \nabla^{k+1} g_t||_{X_T},
\end{equation}
\begin{equation}
B_3 \leq \sum_{\ell=2}^{k} || \nabla^\ell g_t||_{L^\infty(\R^n\times[0,T])}|| \nabla^{k+2 - \ell} g_t ||_{L^\infty(\R^n\times [0,T])} ||1||_{Y^1_T} \leq (k-1)M_k^2\sqrt{T},
\end{equation}
and
\begin{equation}
\begin{split}
B_4 &\leq  \sum_{j=3}^{k+2} \left(\sum_{(i_1,\ldots, i_j)\in P(k+2, j)} ||\nabla^{i_1} g_t||_{C^0(\R^n\times[0,T])} \cdots||\nabla^{i_j} g_t||_{C^0(\R^n\times[0,T])} || 1 ||_{Y^1_T}\right)
\\& \leq c(k)kM_k^{k+2}\sqrt{T}.
\end{split}
\end{equation}

Combining the previous equations, we find
\begin{equation}
\begin{split}
|| \nabla^{k+1} g_t||_{X_T} & \leq  || \nabla^{k+1} g_0||_{L^\infty(\R^n)} 
\\& + C\varepsilon' || \nabla^{k+1} g_t||_{X_T} + M_kT||\nabla^{k+1} g_t||_{X_T} + (k-2)M_k^2 T + (k-1)^2M_k^3T + c(k)M_k^{k+3}T
\\& + C\varepsilon' ||\nabla^{k+1}g_t||_{X_T} + M_k\sqrt{T}|| \nabla^{k+1} g_t||_{X} + (k-1)M_k^2\sqrt{T} + c(k)kM_k^{k+2}\sqrt{T}.
\end{split}
\end{equation}
Choose $T$ sufficiently small depending on $M_k$ so that $M_kT < 1/8$ and $M_k\sqrt{T} < 1/8$. Then, after rearranging the previous equation, we have, by (\ref{eq:Cepsilonsmall})
\begin{equation}
\begin{split}
|| \nabla^{k+1} g_t||_{L^\infty(\R^n\times [0,T])} & \leq || \nabla^{k+1} g_t||_{X_T} 
\\& \leq 2 \big[|| \nabla^{k+1} g||_{L^\infty(\R^n)} + (k-2)M_k^2 T + (k-1)^2M_k^3T 
\\& + c(k)M_k^{k+3}T + (k-1)M_k^2\sqrt{T} + c(k)kM_k^{k+2}\sqrt{T}\big].
\end{split}
\end{equation}
This proves the statement of the lemma for $k+1$, with $T_{k+1} := T$ and $M_{k+1} := 2N + 2(k-2)M_k^2 T_{k+1} + 2(k-1)^2M_k^3T_{k+1}  + 2c(k)M_k^{k+3}T_{k+1} + 2(k-1)M_k^2\sqrt{T_{k+1}} + 2c(k)kM_k^{k+2}\sqrt{T_{k+1}}$.
\end{proof}

\begin{proof}[Proof of Lemma \ref{lemma:SmoothConvergenceae}]
Let $\varepsilon'' = \varepsilon '$, where $\varepsilon '$ is as in Lemma \ref{lemma:smoothRDTFderivbounds}. We will reduce $\varepsilon''$ as needed in the course of the proof. Let $g$ and $U$ be as in the statement of the lemma. Let $K = \overline{U}$. Suppose that there exist no such $C_k$ and $T_k$, so that for all $i\in \mathbb{N}$ there exist $x_i \in K$ and $t_i \in [0, \tfrac{1}{i^{2/k}}]$ such that $|\nabla^k g_{t_i}|(x_i) > i$. To simplify notation, let $c_i = i^{1/k}$. Let $\hat g^i(x) := g(x/c_i + x_i)$ and let $\hat g^i(x,t) := g(x/c_i+ x_i, t/c_i^2)$. Then $g^i$ is $(1+\varepsilon'')$-bilipschitz to $\delta$ on $\R^n$. Also, by Remark \ref{rmk:parabolicscaling}, Remark \ref{rmk:Xnormscalinginvariance}, and the uniqueness statement in Theorem \ref{theorem:KL}, $g^i_t - \delta$ is the solution to the integral equation (\ref{eq:Integraleq}) with initial data $g^i-\delta$ whose existence is guaranteed by Theorem \ref{theorem:KL}. Also by scaling, $|\nabla^k g^i_{c_i^2t_i}|(0) > 1$, where here $0$ denotes the origin. By the derivative estimates in Theorem \ref{theorem:KL} and Arzel\'a -- Ascoli, there exists a subsequence $\{g_t^{i_m}\}$ of $\{g^i_t\}$ that converges in $C^\infty_\loc(\R^n\times (0,\infty))$ to some Ricci-DeTurck flow $(g^\infty_t)_{t>0}$.

Let $r = \dist(K, \Sing) > 0$. Let $V = \{x: \dist(x, K) < r/2\}$ and let $W = \{x: \dist(x, K) < 3r/4\}$. Note that $\overline{V}$ is compact and $W$ is an open set containing $\overline{V}$ so there exists a smooth cutoff function $\chi: \R^n\to [0,1]$ identically equal to $1$ on $\overline{V}$ and identically equal to $0$ on $\R^n\setminus W$  such that for all $j\in \mathbb{N}$, $|| \nabla^j \chi ||_{C^0(\R^n)} \leq c_j(r) = c_j(\dist(K, \Sing))$. Also, $\overline{W}$ is compact, and $\dist(\overline W, \Sing) \geq \dist(K, \Sing) - \dist(\overline{W}, K) \geq r - 3r/4 = r/4 >0$. Let $g^{V} := \chi g + (1-\chi)\delta$. Then $g^V$ is smooth with uniformly bounded derivatives (where the upper bound for $|\nabla^k g^V|$ depends on $\dist(K, \Sing)$ and $|| g||_{C^k(\overline{W})}$), and  $(1+\varepsilon'')$-bilipschitz to $\delta$. By Theorem \ref{theorem:KL} there is a solution $(h_t^V)_{t>0}$ to the integral equation (\ref{eq:Integraleq}) with initial data $g^V - \delta$, and by Lemma \ref{lemma:smoothRDTFderivbounds}, for all $j \in \mathbb{N}$, there exist $M_j = M_j(n, || g - \delta||_{C^j(V)}, \dist(V, \Sing))$ and $T_j = T_j(n, || g - \delta||_{C^j(V)}, \dist(V, \Sing))$ such that for all $t\in [0, T_j]$, 
\begin{equation}\label{eq:compactapproximatorbound}
|| g^V_t - \delta ||_{C^j(\R^n)} \leq M_j,
\end{equation} 
where $g^V_t := h^V_t + \delta$. Now, as before, consider $\hat g^i(x) := g^V(x/c_i + x_i)$ and the corresponding Ricci-DeTurck flow $\hat g^i(x,t) := g^V(x/c_i + x_i, t/c_i^2)$. Again by the derivative estimates in Theorem \ref{theorem:KL}, there exists a subsequence $\{\hat g^{i_{m_\ell}}_t\}$ of $\{\hat g^{i_m}_t\}$ that converges in $C^\infty_\loc(\R^n\times (0,\infty))$ to some Ricci-DeTurck flow $(\hat g^\infty_t)_{t>0}$. Relabel the sequences so that $\{g^i_t\}$ denotes $\{g^{i_{m_\ell}}_t\}$ and $\{\hat g^i_t\}$ denotes $\{\hat g^{i_{m_\ell}}_t\}$. 

We now show that $\hat g^\infty_t = g^\infty_t$ for all $t>0$. First, let $V' = \{x : \dist(x, K)< r/4\}$, and let $\varphi: \R^n\to [0,1]$ be a smooth cutoff function identically equal to $1$ on $\overline{V'}$ and identically equal to $0$ outside of $V$, and such that $|| \varphi||_{C^2(\R^n)} \leq c(r) = c(\dist(K, \Sing))$. Note that since $\chi \equiv 1$ on $V$, we have $\varphi(g - \hat g) \equiv 0$. 

Consider $\varphi^i: \R^n\to \R$ defined by $\varphi^i(x) = \varphi(x/c_i + x_i)$ so that $\varphi^i \equiv 1$ on $B(0, c_ir/4)$ (since $B(x_i, r/4)\subset V'$),  $|| \nabla \varphi^i||_{C^0(\R^n)}\leq c(\dist(K,\Sing))/c_i$, $|| \nabla^2 \varphi^i||_{C^0(\R^n)}\leq c(\dist(K,\Sing))/c_i^2$, and $\varphi^i(g^i - \hat g^i) \equiv 0$. We will show that
\begin{equation}
||\varphi^i(g_t^i - \hat g_t^i) ||_{X_T} \xrightarrow[i\to\infty]{} 0.
\end{equation}
To see why this is true, we argue as in the proof of \cite[Theorem 4.1]{PBG19}: first, note that 
\begin{equation}
\begin{split}
(\partial_t - \Delta)[\varphi^i(g_t^i - \hat g_t^i)] &= \varphi^i[\nabla g_t^i*\nabla g_t^i - \nabla \hat g_t^i*\nabla \hat g_t^i] + \varphi^i \nabla^*[g_t^i * \nabla g_t^i - \hat g_t^i * \nabla \hat g_t^i]
\\&  - (\Delta \varphi^i)(g_t^i - \hat g_t^i) - (\nabla \varphi^i)*\nabla (g_t^i - \hat g_t^i)
\\&= \varphi^i[\nabla g_t^i*\nabla g_t^i - \nabla \hat g_t^i*\nabla \hat g_t^i] +  \nabla^*[\varphi^i(g_t^i * \nabla g_t^i - \hat g_t^i * \nabla \hat g_t^i)]
\\& - (\nabla \varphi^i)* (g_t^i * \nabla g_t^i - \hat g_t^i * \nabla \hat g_t^i) - (\Delta \varphi^i)(g_t^i - \hat g_t^i) - (\nabla \varphi^i)*\nabla (g_t^i - \hat g_t^i).
\end{split}
\end{equation}
Fix some $T>0$. Then, by \cite[Lemma 4.2]{KochLamm12} and arguing similarly to the proofs of \cite[Lemma 4.1]{KochLamm12} and \cite[Theorem 4.1]{PBG19} we have
\begin{align*}
|| \varphi^i(g_t^i - \hat g_t^i) ||_{X_T} & \leq c(n)||\varphi^i(g^i - \hat g^i) ||_{L^\infty(\R^n)}
\\& +C(|| g_t^i - \delta ||_{X_T} + || \hat g_t^i - \delta ||_{X_T})|| \varphi^i(g_t^i - \hat g_t^i)  ||_{X_T}
\\& + c(n)|| \nabla \varphi^i||_{C^0(\R^n)}(|| g_t^i - \delta||_{X_T}^2 + || \hat g^i_t - \delta||_{X_T}^2) 
\\& + c(n)||\Delta \varphi^i ||_{C^0(\R^n)}(|| g_t^i - \delta||_{X_T} + || \hat g^i_t - \delta||_{X_T}) T
\\& + c(n)||\nabla \varphi^i ||_{C^0(\R^n)}(|| g_t^i - \delta||_{X_T} + || \hat g^i_t - \delta||_{X_T}) \sqrt{T}
\\& \leq 0 + 2C\varepsilon''|| \varphi^i(g_t^i - \hat g_t^i)  ||_{X_T} 
\\& + 2c(n)C(\varepsilon'' + (\varepsilon'')^2)\left (\frac{c(\dist(K,\Sing))\sqrt{T}}{c_i} + \frac{c(\dist(K,\Sing))T}{c_i^2} \right),
\end{align*}
where $C$ is the product of constants from \cite[Lemmata 4.1 and 4.2]{KochLamm12}. Therefore, if $\varepsilon''$ is sufficiently small depending on $C$ (and hence on $n$), then we may rearrange this to find that for all $0 < t < T$,
\begin{equation}
\begin{split}
|| g_t^i - \hat g_t^i ||_{C^0(B(0, c_i r/4))} &\leq || \varphi^i(g_t^i - \hat g_t^i) ||_{X_T} 
\\& \leq \frac{2C(\varepsilon'' + (\varepsilon'')^2)}{1 - 2C\varepsilon''}\left (\frac{c(\dist(V,\Sing))}{c_i} + \frac{c(\dist(V,\Sing))}{c_i^2} \right) \xrightarrow[i\to \infty]{} 0,
\end{split}
\end{equation}
and hence 
\begin{equation}\label{eq:RDTFsamelimits}
g_t^\infty \equiv \hat g_t^\infty
\end{equation} 
for all $t > 0$. 

By scaling and the estimate (\ref{eq:compactapproximatorbound}),
\begin{equation*}
|| \nabla^j \hat g^i_t ||_{C^0(\R^n \times [0, c_i^2T_j])} \leq \frac{M_j}{c_i^j}
\end{equation*}
for all $i,j\in \mathbb{N}$ so
\begin{equation*}
|| \nabla^j \hat g^\infty_t ||_{C^0(\R^n \times [0, \infty))} = 0
\end{equation*}
for all $j\in \mathbb{N}$. On the other hand, the $0 < c_i^2t_i \leq c_i^2(1/c_i^2)$ subconverge to some $t_\infty\in [0, 1]$ so after passing to the corresponding subsequences of $\{g^i_t\}$ and $\{\hat g^i_t\}$ we have
\begin{align*}
0 &= |\nabla^k \hat g^\infty_t|(0) = |\nabla^k g^{\infty}_{t_\infty}|(0) \geq |\nabla^k g^i_{c_i^2t_i}|(0) - | \nabla^k g^i_{c_i^2t_i} - \nabla^k g^i_{t_\infty} |(0) - |\nabla^k g^i_{t_\infty} - \nabla^k g^\infty_{t_\infty}|(0)
\\& = |\nabla^k g^i_{c_i^2t_i}|(0) - | \nabla^k \hat g^i_{c_i^2t_i} - \nabla^k \hat g^i_{t_\infty} |(0) - |\nabla^k g^i_{t_\infty} - \nabla^k g^\infty_{t_\infty}|(0)
\\& \geq 1 - \frac{c(M_{k+2})}{c_i^{k+2}}|c_i^2t_i - t_\infty| - |\nabla^k g^i_{t_\infty} - \nabla^k g^\infty_{t_\infty}|(0) \xrightarrow[i\to \infty]{} 1,
\end{align*}
where again we have used scaling, (\ref{eq:compactapproximatorbound}), (\ref{eq:RDTFsamelimits}), and in the last inequality we have used the evolution equation for $\nabla^k \hat g_t^i$. This contradiction proves the lemma.
\end{proof}

\bibliographystyle{plain}
\bibliography{NNSCLinftyRDTFbiblio}

\end{document}